\documentclass[a4paper,11pt,oneside]{article}
\usepackage[utf8]{inputenc}

\usepackage{amsmath,amssymb}
\usepackage{amsthm, mathtools}
\usepackage{amsfonts}
\usepackage{rotating}

\usepackage{tikz}
\usepackage{pgfplots}

\usepackage[english]{babel} 
\usepackage[utf8]{inputenc}
\usepackage{geometry}
\usepackage{multicol}
\usepackage{graphicx} 
\usepackage{color}

\title{Coupled BEM-FEM for the convected Helmholtz equation with non-uniform flow
in a bounded domain}
\author{Fabien Casenave$^{1,2}$, Alexandre Ern$^{1}$ and Guillaume Sylvand$^{2}$\\~\\
$^1$ Universit\'{e} Paris-Est, CERMICS (ENPC), 6-8 Avenue Blaise Pascal, Cit\'{e} Descartes,\\
F-77455 Marne-la-Vall\'{e}e, France\\
$^2$ EADS-IW, 18 rue Marius Terce, 31300 Toulouse, France}

\newlength\figureheight
\newlength\figurewidth

\begin{document}

\maketitle

\def\N {\mathbb{N}}
\def\Z {\mathbb{Z}}
\def\R {\mathbb{R}}
\def\C {\mathbb{C}}
\def\P {\mathbb{P}}

\newcommand{\nc}{\newcommand}

\nc{\dsp}{\displaystyle}
\nc{\y}{\mathbf{y}}
\nc{\dt}{\mathit{\Delta t}}
\nc{\dx}{\mathit{\Delta x}}
\nc{\dy}{\mathit{\Delta y}}
\nc{\PRe}{\Re e}

\nc{\ovra}{\overrightarrow}
\nc{\ud}{\mathrm{d}}

\nc{\xt}{(\vec{x},t)}
\nc{\rz}{(r,z)}
\nc{\rzte}{(r,z,\theta)}
\nc{\nablav}{\ovra{\nabla}}
\nc{\x}{\vec{x}}
\nc{\n}{\vec{n}}
\nc{\V}{\vec{v}}
\nc{\vinf}{v_{\infty}}
\nc{\ainf}{a_{\infty}}
\nc{\pas}{{\psi_a^*}}
\nc{\pam}{{\psi_{am}^*}}
\nc{\Vas}{\vec{v_a^*}}
\nc{\kos}{{k_0}}

\nc{\er}{\vec{e}_r}
\nc{\ez}{\vec{e}_z}
\nc{\ete}{\vec{e}_{\theta}}

\renewcommand{\vec}[1]{\boldsymbol{#1}}
\nc{\unvec}[1]{\hat{#1}}

\nc{\grad}{\vec{\nabla}}

\nc{\inc}{{\textrm{inc}}}
\nc{\diffr}{{\textrm{diff}}}

\nc{\todo}[1]{\textit{\textbf{todo: #1 !}}}

\newtheorem{theorem}{Theorem}[section]
\newtheorem{definition}[theorem]{Definition}
\newtheorem{proposition}[theorem]{Proposition}
\newtheorem{cor}[theorem]{Corollary}
\newtheorem{lemma}[theorem]{Lemma}
\newtheorem{remark}[theorem]{Remark}
\newtheorem{proofof}[theorem]{Proof}

\nc{\indices}[4]{
\scriptsize
\begin{array}{c}
1\leq #1\leq #2\\
1\leq #3\leq #4
\end{array}
}

\nc{\indiceslight}[2]{
\scriptsize
\begin{array}{c}
#1\\
#2
\end{array}
}

\begin{abstract}
We consider the convected Helmholtz equation modeling linear acoustic propagation at a fixed frequency in a subsonic
flow around a scattering object.
The flow is supposed to be uniform in the exterior domain far from the object, and potential in the interior domain close to the
object. Our key idea is the reformulation of the original problem using the Prandtl--Glauert transformation on the whole flow domain,
yielding (i) the classical
Helmholtz equation in the exterior domain and (ii) an anisotropic diffusive PDE with skew-symmetric first-order perturbation in the
interior domain such that its transmission condition at the coupling boundary naturally fits the Neumann condition from the
classical Helmholtz equation. Then, efficient off-the-shelf tools can be used to perform the BEM-FEM coupling,
leading to two novel variational formulations for the convected Helmholtz equation.
The first formulation involves one surface unknown and can be affected by resonant frequencies, while the second formulation
avoids resonant frequencies and involves two surface unknowns.
Numerical simulations are presented to compare the two formulations.
\end{abstract}

\section{Introduction}


The scope of the present work is the computation of linear acoustic wave propagation at a fixed frequency in the presence of a flow. 
When the flow is at rest, the simplest model is the classical Helmholtz
equation for the acoustic potential. This equation can be reduced to
finding unknown functions defined on the surface of the scattering
object and solving integral equations which can be effectively
approximated by 
the Boundary Element Method (BEM)~\cite{sauter}.
When the medium of propagation is non-uniform, a volumic
resolution has to be considered using, e.g., a Finite Element Method (FEM).
If such non-uniformities occur only in a given bounded domain, it is possible to benefit from
the advantages of both a volumic resolution and an integral equation formulation. Coupling BEM and FEM
at the boundary of the given bounded domain allows this. Coupled BEM-FEM
can be traced back to McDonald and Wexler~\cite{jcp-biblio2},
Zienkiewicz, Kelly and Bettess~\cite{zienkiewicz1}, Johnson and N\'ed\'elec~\cite{Nedelec} and Jin and Liepa~\cite{jcp-biblio1}.
Over the last decade, such methods have been investigated, among others, for
electromagnetic scattering~\cite{rhipt1, levillain, bemfemcoupled3}, elasticity~\cite{bemfemelasticity}, and
fluid-structure~\cite{bemfemcoupled2} or solid-solid interactions~\cite{nonlinear1, nonlinear2}.
Coupled BEM-FEM for the classical Helmholtz equation can present resonant frequencies,
leading to infinitely many solutions. All these solutions deliver the same acoustic
potential in the exterior domain, but the numerical procedure becomes ill-conditioned.
This problem has been solved in~\cite{buffa2, hiptmair}, where a stabilization of the coupling,
based on combined field integral equations (CFIE), has been
proposed by introducing an additional unknown at the coupling surface.

When the medium of propagation is not at rest, the simplest governing equation is the convected Helmholtz
equation resulting from the linearized harmonic Euler equations. 
Nonlinear interaction between acoustics and fluid mechanics is not considered herein; we refer to the 
early work of Lighthill for aerodynamically generated acoustic sources~\cite{lighthill1, lighthill2},
to~\cite{nonlinear} for a review on nonlinear acoustics, and to~\cite{utzmann} for
the coupling of Computational Aero Acoustic (CAA) and Computational Fluid Dynamics (CFD) solvers.
Moreover, we assume that the flow is potential close to the scattering object and uniform far away from it.
This geometric setup leads to a partition of the unbounded medium of propagation into two subdomains, the bounded
interior domain near the scattering object where the flow is non-uniform and the unbounded exterior domain far away from
the object where the flow is uniform.
The main contribution of this work is the reformulation of the convected Helmholtz equation using the Prandtl--Glauert transformation
on the whole flow domain, yielding (i) the classical
Helmholtz equation in the exterior domain and (ii) an anisotropic diffusive PDE with skew-symmetric first-order perturbation in the
interior domain such that its transmission condition at the coupling boundary naturally fits the Neumann condition from the
classical Helmholtz equation. The Prandtl--Glauert transformation has been used in~\cite{DDMT} for
the uniformly convected Helmholtz equation. In the present case where the flow is non-uniform in the interior domain,
this reformulation allows us to use efficient off-the-shelf tools to perform a BEM-FEM coupling. Namely, a FEM is utilized in the
interior domain to discretize the anisotropic second-order PDE, a BEM is utilized for the classical Helmholtz equation in the
exterior domain, and Dirichlet-to-Neumann maps are used for the coupling.
We emphasize that the key advantage of using the Prandtl--Glauert transformation is that
the BEM part of the resolution only involves integral operators corresponding to the classical Helmholtz equation.
We consider two approaches for the coupling, leading,
to the authors' knowledge, to two novel coupled BEM-FEM formulations for the convected Helmholtz equation.
The first formulation involves one surface unknown and can be affected by resonant frequencies, while the second one
uses the stabilized CFIE technique from~\cite{buffa2, hiptmair},
avoids resonant frequencies, and involves two surface unknowns.
Our numerical results show that the first formulation yields results polluted by spurious oscillations in the 
close vicinity of resonant frequencies, whereas the second formulation yields consistent solutions at all frequencies.
This advantage of the second formulation is particularly relevant in practice at high frequencies, where the density of
resonant frequencies is higher.

We briefly discuss alternative methods from the literature to solve the convected Helmholtz equation in unbounded domains.
In some cases with simple geometries, the far-field solution is
analytically known~\cite{powles}. Boundary integral equations involving
the Green kernel associated with the convected Helmholtz equation have 
been derived in~\cite{jcp-biblio6}. Other numerical methods include
infinite finite elements~\cite{bettess, zienkiewicz} and the method of
fundamental 
solutions~\cite{fairweather}.
An alternative approach to treat unbounded domains is to use Perfectly
Matched Layers (PML), combined with a volumic resolution 
using, e.g., the FEM. Versions of PML for the convected Helmholtz equation
are considered in~\cite{jcp-biblio3, jcp-biblio5}. The use of PML allows one to consider unbounded domains of propagation, but
the solution is only available within the domain of computation. This can be a drawback in the following situations:
(i) when one is interested in the pressure field far away from the scattering object, or (ii) when scattering objects are located
far away from each other so that the volumic resolution has to be carried out in a very large area. Instead,
with coupled BEM-FEM, the volumic resolution only takes place in the areas where the flow is non-uniform, and the pressure
can be retrieved at any point of the exterior domain using known representation formulae, regardless of the distance of this point to the
scattering objects.
However, coupled BEM-FEM exhibit matrices with dense blocks for the unknowns on the boundary,
and an additional treatment is sometimes needed to avoid resonant
frequencies. These two points are addressed in this work.
%
%
%



The material is organized as follows:
off-the-shelf tools useful to carry out the coupling are recalled in Section~\ref{sec:basic_ing}. 
The Prandtl--Glauert transformation of the convected Helmholtz equation is derived in Section~\ref{sec:aeroacouspb}.
The coupled variational formulations are obtained in Section~\ref{sec:coupling}, and the most salient points in their mathematical
analysis are presented.
The finite-dimensional approximation of the coupled formulations is addressed in Section~\ref{sec:findim}, along with a
discussion on the structure of the linear systems and the algorithms to solve them effectively.
Finally, numerical results are presented in Section~\ref{sec:numres}, and some conclusions are drawn in
Section~\ref{sec:conclusion}.

\section{Classical tools for BEM-FEM coupling of the classical Helmholtz equation}
\label{sec:basic_ing}

In this section, the ingredients used to carry out the BEM-FEM coupling are recalled in the context of the
classical Helmholtz equation (so that the medium of propagation is at rest).

\subsection{Boundary integral operators}
\label{sec:derivofdtn}

\setlength\figureheight{0.43\textwidth}
\setlength\figurewidth{0.45\textwidth}
\begin{figure}[h!]
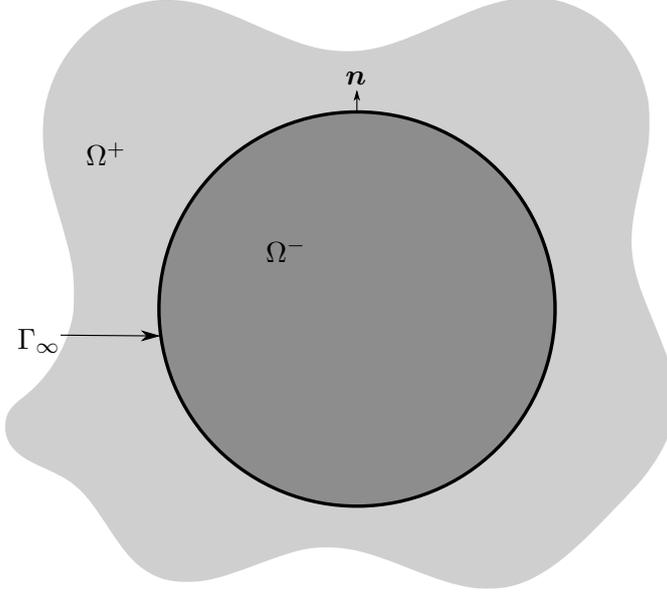

 \centering
\include{conventions0}
 \caption{Geometric setup for the classical Helmholtz equation}
\label{fig:geoconv0}
\end{figure}

Figure~\ref{fig:geoconv0} describes the geometric setup considered in
this section. Let $\Omega^{-}$ be a bounded open set with boundary
$\Gamma_\infty$, and set $\Omega^+:=\mathbb{R}^3\backslash\overline{\Omega^-}$.
The surface $\Gamma_\infty$ is assumed to be Lipschitz.
The one-sided Dirichlet traces on $\Gamma_\infty$ of a smooth function $u$ in $\Omega^+\cup\Omega^-$ are defined as
$\gamma_0^{\pm}u^\pm=u^{\pm}|_{\Gamma_\infty}$, and
the one-sided Neumann traces as $\gamma_1^{\pm}u^\pm=\left(\vec{\nabla}u^\pm\right)|_{\Gamma_\infty}\cdot\vec{n}$, where $u^\pm=u|_{\Omega^\pm}$
and where $\vec{n}$ is the unit normal vector to $\Gamma_\infty$ conventionally pointing towards $\Omega^+$.
These trace operators are extended to bounded linear operators
$\gamma_0^\pm:H^1(\Omega^{\pm})\rightarrow H^{\frac{1}{2}}(\Gamma_\infty)$ and $\gamma_1^\pm:H(\Delta, \Omega^{\pm})\rightarrow H^{-\frac{1}{2}}(\Gamma_\infty)$
, where $H^{\frac{1}{2}}(\Gamma_\infty)$ and $H^{-\frac{1}{2}}(\Gamma_\infty)$ are
the usual Sobolev spaces on $\Gamma_\infty$, and
$H(\Delta, \Omega^{\pm}):=\{v\in H^1(\Omega^{\pm}),\Delta v\in  L^2(\Omega^{\pm})\}$
(see~\cite[Lemma 20.2]{Tartar}).
It is actually sufficient to consider functional spaces on compact subsets of $\Omega^+$ to define exterior traces on $\Gamma_\infty$.
For $u\in H^1(\Omega^+\cup\Omega^-)$, the jump and average of its Dirichlet traces across $\Gamma_\infty$ are defined respectively as
$\left[\gamma_0 u\right]_{\Gamma_\infty}=\gamma_0^{+}u^+-\gamma_0^{-}u^-$ and
$\left\{\gamma_0 u\right\}_{\Gamma_\infty}=\frac{1}{2}\left(\gamma_0^{+}u^++\gamma_0^{-}u^-\right)$. 
For $u\in H(\Delta, \Omega^+\cup\Omega^-):=\{v\in H^{1}(\Omega^+\cup\Omega^-), \Delta v\in L^2(\Omega^+\cup\Omega^-)\}$,
the jump and average of its Neumann traces across $\Gamma_\infty$ are defined respectively as
$\left[\gamma_1 u\right]_{\Gamma_\infty}=\gamma_1^{+}u^+-\gamma_1^{-}u^-$ and
$\left\{\gamma_1 u\right\}_{\Gamma_\infty}=\frac{1}{2}\left(\gamma_1^{+}u^++\gamma_1^{-}u^-\right)$.
When a trace is single-valued on $\Gamma_\infty$, we omit the superscripts $\pm$.
Furthermore, the $L^2(\Gamma_\infty)$-inner product
$\langle\cdot,\cdot\rangle_{L^2(\Gamma_\infty), L^2(\Gamma_\infty)}:L^2(\Gamma_\infty)\times L^2(\Gamma_\infty)\rightarrow \mathbb{C}$ is defined as
$\langle\lambda,\mu\rangle_{L^2(\Gamma_\infty),
  L^2(\Gamma_\infty)}=\int_{\Gamma_\infty}\overline{\lambda}(\vec{y}){\mu}(\vec{y})ds(\vec{y})$,
where $\overline{\cdot}$ denotes the complex conjugate,
and is extended to a duality pairing on $H^{-\frac{1}{2}}(\Gamma_\infty)\times H^{\frac{1}{2}}(\Gamma_\infty)$
denoted by
$\langle\cdot,\cdot\rangle_{H^{-\frac{1}{2}}(\Gamma_\infty),H^{\frac{1}{2}}(\Gamma_\infty)}$. We
then define the product
\begin{equation}
\left(\lambda,\mu\right)_{\Gamma_\infty}=\left\{
\begin{alignedat}{2}
\langle\lambda,\mu\rangle_{H^{-\frac{1}{2}}(\Gamma_\infty),H^{\frac{1}{2}}(\Gamma_\infty)}&\qquad\textnormal{if }\lambda\in H^{-\frac{1}{2}}(\Gamma_\infty)&,&\quad\mu\in H^{\frac{1}{2}}(\Gamma_\infty),\\
\overline{\langle\mu,\lambda\rangle}_{H^{-\frac{1}{2}}(\Gamma_\infty),H^{\frac{1}{2}}(\Gamma_\infty)}&\qquad\textnormal{if }\lambda\in H^{\frac{1}{2}}(\Gamma_\infty)&,&\quad\mu\in H^{-\frac{1}{2}}(\Gamma_\infty).
\end{alignedat}
\right.
\end{equation}

Consider the following equations:
\begin{subequations}
\label{eq:rphs}
 \begin{align}
\Delta u + k^2 u &= 0\quad\textnormal{in }{\Omega^+\cup\Omega^-},\label{eq:rphs1}\\
\lim_{r\rightarrow +\infty} r\left( \frac{\partial{u}}{\partial{r}}-i k u \right)&=0,\label{eq:rphs2}
 \end{align}
\end{subequations}
where $k$ is the wave number.
A function solving~\eqref{eq:rphs1} is said to be a piecewise Helmholtz solution
and a function solving~\eqref{eq:rphs1}-\eqref{eq:rphs2} is said to be a radiating piecewise Helmholtz solution.
The condition at infinity~\eqref{eq:rphs2} is the Sommerfeld radiation condition,
that guarantees existence and uniqueness for Helmholtz exterior problems~\cite[Theorem 9.10]{mclean}.

For all $\lambda\in C^0(\Gamma_\infty)$, the single-layer potential is defined as $\mathcal{S}(\lambda)(\vec{x})=\int_{\Gamma_\infty}E(\vec{y}-\vec{x})
\lambda(\vec{y})ds(\vec{y})$, $\vec{x}\in\mathbb{R}^3\backslash\Gamma_\infty$, where $E(\vec{x})=\frac{\exp(i{k}\left|\vec{x}\right|)}{4\pi \left|\vec{x}\right|}$
is the fundamental solution of~\eqref{eq:rphs1}-\eqref{eq:rphs2}. For all $\mu\in C^0(\Gamma_\infty)$, the double-layer potential is defined as $\mathcal{D}(\mu)
(\vec{x})=\int_{\Gamma_\infty}\vec\nabla_{\vec{y}} E(\vec{y}-\vec{x})\mu(\vec{y})ds(\vec{y})$, $\vec{x}\in\mathbb{R}^3\backslash\Gamma_\infty$.
From~\cite[Theorem 3.1.16]{sauter}, these operators can be extended to bounded linear operators
$\mathcal{S}:H^{-\frac{1}{2}}(\Gamma_\infty)\rightarrow H^1_{\rm loc}(\mathbb{R}^3)$
and $\mathcal{D}:H^{\frac{1}{2}}(\Gamma_\infty)\rightarrow H^1_{\rm loc}(\mathbb{R}^3\backslash\Gamma_\infty)$
where, for any open set $X$, $H^{1}_{\rm loc}(X)=\{u\in H^{1}(K),\forall K\subset X\textnormal{ compact}\}$.
Moreover, both operators map onto radiating piecewise Helmholtz solutions.
Recalling~\cite[Theorem 3.1.1]{Nedelec}, a radiating piecewise Helmholtz solution $u$ can be represented from its Dirichlet and
Neumann jumps across $\Gamma_\infty$ in the form
\begin{equation}
\label{eq:reptheo}
u=-\mathcal{S}([\gamma_1 u]_{\Gamma_\infty})+\mathcal{D}([\gamma_0 u]_{\Gamma_\infty})\quad\textnormal{in }\Omega^+\cup\Omega^-.
\end{equation}
The operators
\begin{equation}
\label{eq:intboundop}
\begin{alignedat}{2}
S&:H^{-\frac{1}{2}}(\Gamma_\infty)\rightarrow H^{\frac{1}{2}}(\Gamma_\infty),&\quad S\lambda&=\gamma_0\left(\mathcal{S}\lambda\right),\\
D&:H^{\frac{1}{2}}(\Gamma_\infty)\rightarrow H^{\frac{1}{2}}(\Gamma_\infty),&\quad D\mu&=\left\{\gamma_0\left(\mathcal{D}\mu\right)\right\}_{\Gamma_\infty},\\
\tilde{D}&:H^{-\frac{1}{2}}(\Gamma_\infty)\rightarrow H^{-\frac{1}{2}}(\Gamma_\infty),&\quad \tilde{D}\lambda&=\left\{\gamma_1\left(\mathcal{S}\lambda\right)\right\}_{\Gamma_\infty},\\
N&:H^{\frac{1}{2}}(\Gamma_\infty)\rightarrow H^{-\frac{1}{2}}(\Gamma_\infty),&\quad N\mu&=-\gamma_1\left(\mathcal{D}\mu\right),
\end{alignedat}
\end{equation}
are respectively the single-layer, double-layer, transpose (or dual) of the double-layer, and hypersingular boundary integral operators.
The Dirichlet and Neumann traces are well-defined, and the functional setting can be found in~\cite[Theorem 7.1]{mclean}.
From~\cite[Theorem 3.1.2]{Nedelec}, if $u$ is a radiating piecewise Helmholtz solution, there holds
\begin{equation}
\label{eq:cald}
\begin{pmatrix}
\begin{array}{cc}
\frac{1}{2}I-D& S\\
 N& \frac{1}{2}I+\tilde{D}
\end{array}\end{pmatrix}\begin{pmatrix}
\begin{array}{cc}
[{\gamma_0}u]_{\Gamma_\infty}\\
[{\gamma_1}u]_{\Gamma_\infty}
\end{array}\end{pmatrix}=-\begin{pmatrix}
\begin{array}{cc}
\gamma_0^- u^-\\
\gamma_1^- u^-
\end{array}
\end{pmatrix}.
\end{equation}

\subsection{Transmission problems}

Consider the following transmission problem:
\begin{subequations}
\label{eq:transmipb}
 \begin{align}
F(u)&=0\quad\textnormal{in }{\Omega^-},\label{eq:transmipb1}\\
\Delta (u-u_{\rm inc}) + k^2 (u-u_{\rm inc}) &= 0\quad\textnormal{in }{\Omega^+},\label{eq:transmipb2}\\
{\gamma_0^+}u^+-{\gamma_0^-}u^-&=0\quad\textnormal{on }\Gamma_\infty,\label{eq:transmipb4}\\
{\gamma_1^+}u^+-{\gamma_1^-}u^-&=0\quad\textnormal{on }\Gamma_\infty,\label{eq:transmipb5}\\
\lim_{r\rightarrow +\infty} r\left( \frac{\partial{(u-u_{\rm inc})}}{\partial{r}}-i {k} (u-u_{\rm inc}) \right)&=0,\label{eq:transmipb6}
 \end{align}
\end{subequations}
where $F$ denotes some differential operator and $u_{\rm inc}$ is an incident acoustic field.
The field $u_{\rm inc}$ is created by a source located in~$\Omega^+$, and solves the classical Helmholtz equation outside
the support of this source. In particular,
\begin{equation}
\label{eq:finc0}
\Delta u_{\rm inc} + k^2 u_{\rm inc} = 0\quad\textnormal{in }{\Omega^-}.
\end{equation}
Let now $u$ solve~\eqref{eq:transmipb} and let $v$ be the function defined by $v|_{\Omega^+}=u-u_{\rm inc}$ and
$v|_{{\Omega^-}}=-u_{\rm inc}$. The function $v$ is a radiating piecewise Helmholtz solution
(this follows from~\eqref{eq:transmipb2} and~\eqref{eq:transmipb6} on $\Omega^+$, and from~\eqref{eq:finc0} on ${\Omega^-}$).
Moreover, since the field $u_{\rm inc}$ is continuous across $\Gamma_\infty$,
\begin{equation}
\left[\gamma_0v\right]_{\Gamma_\infty}=\gamma_0^+(u-u_{\rm inc})^++\gamma_0^-u^-_{\rm inc}=
\gamma_0^+(u-u_{\rm inc})^++\gamma_0^+u^+_{\rm inc}={\gamma_0^+}u^+.
\end{equation}
Likewise, $\left[\gamma_1v\right]_{\Gamma_\infty}={\gamma_1^+}u^+$.
Then,~\eqref{eq:cald} applied to $v$ yields
\begin{equation}
\label{eq:DtNsys}
\begin{pmatrix}
\begin{array}{cc}
\frac{1}{2}I-D&S\\
N&\frac{1}{2}I+\tilde{D}
\end{array}\end{pmatrix}\begin{pmatrix}
\begin{array}{cc}
{\gamma_0}u\\
{\gamma_1}u
\end{array}\end{pmatrix}=\begin{pmatrix}
\begin{array}{cc}
\gamma_0u_{\rm inc}\\
\gamma_1u_{\rm inc}
\end{array}
\end{pmatrix}.
\end{equation}
A $DtN$ operator maps any function $\theta$ to the Neumann trace $\gamma_1 u$ where $u$ solves
the exterior Helmholtz problem, with $\gamma_0 u=\theta$ as Dirichlet boundary condition on $\Gamma_\infty$.
Various $DtN$ maps can be derived from~\eqref{eq:DtNsys}.
Two examples are detailed in Sections~\ref{sec:unstabdtn} and~\ref{sec:stabdtn} below.



\subsection{An unstable DtN map}
\label{sec:unstabdtn}

Using the first line of~\eqref{eq:DtNsys}, $\gamma_1u=S^{-1}\left(\left(D-\frac{1}{2}I\right)\left(\gamma_0u\right)
+\gamma_0u_{\rm inc}\right)$. At this point, the inverse of $S$ is written formally. Conditions of inversibility are discussed below.
From the second line of~\eqref{eq:DtNsys}, $\gamma_1u=-N(\gamma_0u)+\left(\frac{1}{2}I-{\tilde{D}}\right)
\left(\gamma_1u\right)+\gamma_1u_{\rm inc}$. Injecting into the right-hand side of this relation the expression of
$\gamma_1u$ derived above yields the $DtN$ affine map:
$DtN_{\rm unstab}:H^{\frac{1}{2}}(\Gamma_\infty)\rightarrow H^{-\frac{1}{2}}(\Gamma_\infty)$ such that
\begin{equation}
\label{eq:DtNunst}
DtN_{\rm unstab}(\gamma_0u)=-N(\gamma_0u)+\left(\frac{1}{2}I-{\tilde{D}}\right)(\lambda)+\gamma_1u_{\rm inc},
\end{equation}
where the auxiliary field $\lambda\in H^{-\frac{1}{2}}(\Gamma_\infty)$ is such that
\begin{equation}
\label{eq:deflam}
\left(D-\frac{1}{2}I\right)(\gamma_0u) - S\lambda = -{\gamma_0}u_{\rm inc}.
\end{equation}

The main difficulty with the map~\eqref{eq:DtNunst} stems from the fact that $\ker(S)$ depends on
whether $-k^2$ belongs to the set $\Lambda$ of Dirichlet eigenvalues for the Laplacian on the bounded domain
$\Omega^-$. Specifically, $\ker(S)=\left\{0\right\}$ if $-k^2\notin\Lambda$, while
$\ker(S)$ contains nontrivial elements if $-k^2\in\Lambda$.

\begin{remark}
The $DtN_{\rm unstab}$ affine map was proposed by Costabel to obtain a symmetric coupling in the case of self-adjoint operators
\cite{Costabel}. The $DtN_{\rm unstab}$ map can be well-defined for certain operators, for instance for transmission
problems for the Laplace equation, the unstability being here linked to the specificity of the Helmholtz equation.
\end{remark}

\subsection{A stable DtN map}
\label{sec:stabdtn}

The idea of considering a linear combination of $S$ and $\frac{1}{2}I+\tilde{D}$ to derive well-posed boundary integral equations
was independently proposed by Brakhage and Werner~\cite{BW}, Leis~\cite{Leis} and Panich~\cite{Panich}.
This is the so-called Combined Field Integral Equation (CFIE).
However, $S$ and $\tilde{D}$ map $H^{-\frac{1}{2}}(\Gamma_\infty)$ into different spaces ($H^{\frac{1}{2}}(\Gamma_\infty)$ and
$H^{-\frac{1}{2}}(\Gamma_\infty)$ respectively).
This inconsistency in the functional setting can be solved
by considering a regularizing compact operator from $H^{-\frac{1}{2}}(\Gamma_\infty)$ into $H^{\frac{1}{2}}(\Gamma_\infty)$,
as observed by Buffa and Hiptmair~\cite{buffa2}.

We briefly recall the approach of~\cite{buffa2}.
Let $\vec{\nabla}_{\Gamma_\infty}$ denote the surfacic gradient on $\Gamma_\infty$.
Consider the following Hermitian sesquilinear form: For all $p,q\in H^{1}(\Gamma_\infty)$,
\begin{equation}
\label{eq:defc}
\delta_{\Gamma_\infty}(p,q)=\left(\vec{\nabla}_{\Gamma_\infty}p, \vec{\nabla}_{\Gamma_\infty}q\right)_{\Gamma_\infty}+
\left(p,q\right)_{\Gamma_\infty},
\end{equation}
and the regularizing operator $M:H^{-1}(\Gamma_\infty)\rightarrow H^{1}(\Gamma_\infty)$ defined through the following implicit
relation: For all $p\in H^1(\Gamma_\infty)$, 
$\delta_{\Gamma_\infty}(Mp,q)=(p,q)_{\Gamma_\infty}$ for all $q\in H^1(\Gamma_\infty)$.
It is readily seen that $M=(-\Delta_{\Gamma_\infty}+I)^{-1}$, where
$\Delta_{\Gamma_\infty}$ is the Laplace--Beltrami operator on $\Gamma_\infty$.
Many choices of $DtN$ maps based on CFIE strategies with the regularizing operator $M$ lead to well-posed systems whatever
the value of $k$. 
The present choice hinges on the inversion of the operator $S+i\eta M\left(\frac{1}{2}I+{\tilde{D}}\right)$ mapping
$H^{-\frac{1}{2}}(\Gamma_\infty)$ into $H^{\frac{1}{2}}(\Gamma_\infty)$ since,
from~\cite[Lemma 4.1]{buffa2}, this operator is bijective as long as the coupling parameter $\eta$ is such that
${\rm Re}(\eta)\neq 0$.
To do so, the first line of~\eqref{eq:DtNsys} and the application of $M$ to the second line of~\eqref{eq:DtNsys} are used to
obtain
\begin{equation}
\label{eq:DtNsys3}
\begin{pmatrix}
\begin{array}{cc}
\left(\frac{1}{2}I-D\right)+i\eta MN&S+i\eta M\left(\frac{1}{2}I+\tilde{D}\right)\\
N&\frac{1}{2}I+\tilde{D}
\end{array}\end{pmatrix}\begin{pmatrix}
\begin{array}{cc}
\gamma_0u\\
\gamma_1u
\end{array}\end{pmatrix}=\begin{pmatrix}
\begin{array}{cc}
\gamma_0u_{\rm inc}+i\eta M\gamma_1u_{\rm inc}\\
\gamma_1u_{\rm inc}
\end{array}
\end{pmatrix}.
\end{equation}
Then, using both equations in~\eqref{eq:DtNsys3} in the same fashion as in Section~\ref{sec:unstabdtn} leads to
$DtN_{\rm stab}: H^{\frac{1}{2}}(\Gamma_\infty)\rightarrow H^{-\frac{1}{2}}(\Gamma_\infty)$ such that
\begin{equation}
\label{eq:dtnst}
\dsp DtN_{\rm stab}(\gamma_0u)= -N(\gamma_0u)+\left(\frac{1}{2}I-{\tilde{D}}\right)(\lambda)+\gamma_1u_{\rm inc},
\end{equation}
where $\lambda\in H^{-\frac{1}{2}}(\Gamma_\infty)$ is such that
\begin{equation}
\label{eq:lambdast}
S(\lambda)+\left(\frac{1}{2}I-D\right)(\gamma_0u)+i\eta p=\gamma_0u_{\rm inc},
\end{equation}
and $p\in H^{1}(\Gamma_\infty)$ is such that for all $q\in H^1(\Gamma_\infty)$,
\begin{equation}
\label{eq:pst}
\delta_{\Gamma_\infty}(p,q)=\left(N(\gamma_0u),q\right)_{\Gamma_\infty}+\left(\left(\frac{1}{2}I+{\tilde{D}}\right)(\lambda),q\right)_{\Gamma_\infty}-\left(\gamma_1u_{\rm inc},q\right)_{\Gamma_\infty}.
\end{equation}
\section{The aeroacoustic problem}
\label{sec:aeroacouspb}

This section describes the problem of acoustic scattering by a solid object in a non-uniform convective flow,
together with the underlying physical assumptions.

\subsection{Notation and preliminaries}

\setlength\figureheight{0.43\textwidth}
\setlength\figurewidth{0.45\textwidth}
\begin{figure}[h!]
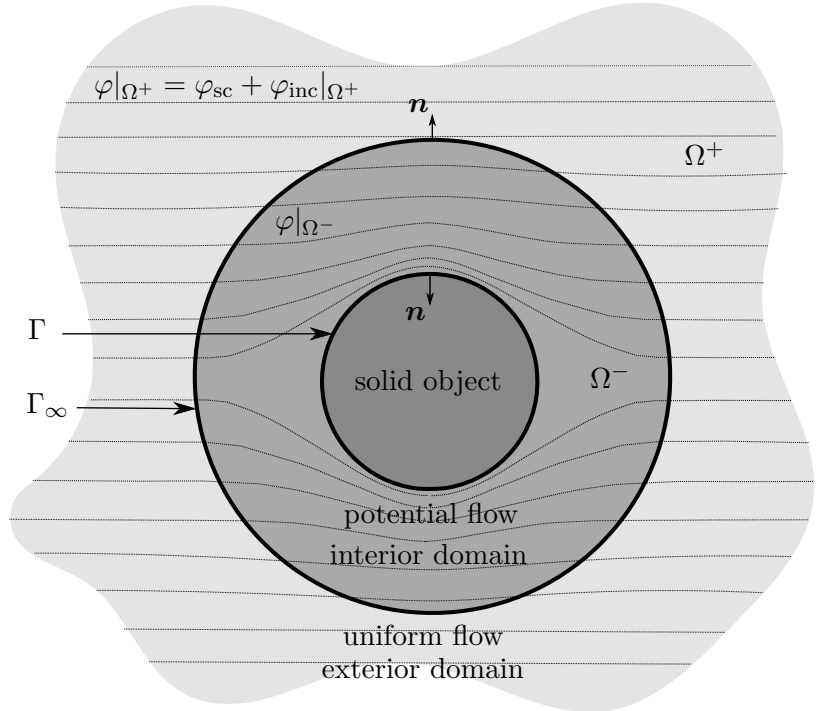

 \centering
\include{conventions2}
 \caption{Geometric setup for the convected Helmholtz equation}
\label{fig:geoconv}
\end{figure}

Figure~\ref{fig:geoconv} describes the geometric setup. The interior domain, corresponding
to the area near the scattering object where the convective flow is non-uniform, is denoted by $\Omega^{-}$.
In the exterior domain, $\Omega^{+}$, the convective flow is assumed to be uniform.  
The complete medium of propagation, denoted by $\Omega\subset\mathbb{R}^3$, is such that
$\Omega=\Omega^{+}\cup\Omega^{-}\cup\Gamma_\infty=\mathbb{R}^3\backslash\overline{\{\textnormal{solid object}\}}$, where
$\Gamma_\infty=\partial{\Omega^+}\cap\partial{\Omega^-}$ is the boundary between the interior and exterior domains.
The surface $\Gamma_\infty$ is assumed to be Lipschitz. Such an assumption is sufficiently large to include for instance
polyhedric surfaces resulting from the use of a finite element mesh in $\Omega^-$.
The surface of the solid scattering object, denoted by
$\Gamma=\partial{\Omega^-}\backslash\Gamma_\infty$, is assumed to be
Lipschitz. The source term $g$ is time-harmonic with pulsation $\omega$ and is assumed to be located in $\Omega^+$. Typically, this source
term can be an acoustic monopole located at $x_s\in\Omega^+$ of amplitude $A_s$, so that $g=A_s\delta_{x_s}$, where
$\delta_{x_s}$ denotes the Dirac mass distribution at $x_s$.

The speed of sound when the medium of propagation is at rest is denoted by $c$, the wave number by $k$, the density by $\rho$,
and the acoustic velocity and pressure, respectively, by $\vec{v}$ and $p$. The rescaled velocity is defined as
$\vec{M}=c^{-1}\vec{v}$, where $M=\left|\vec{M}\right|$ is the Mach number.
The subscript $\infty$ refers to uniform flow quantities related to $\Omega^+$, whereas the subscript $0$ refers to
point-dependent flow quantities related to $\Omega^-$, that is,
$\rho_{|\Omega^-}=\rho_0(\vec{x})$, $\rho_{|\Omega^+}\equiv\rho_\infty$, 
$k_{|\Omega^-}=k_0(\vec{x})$, $k_{|\Omega^+}\equiv k_\infty$,
$c_{|\Omega^-}=c_0(\vec{x})$, $c_{|\Omega^+}\equiv c_\infty$,
$\vec{M}_{|\Omega^-}=\vec{M}_0(\vec{x})$, and $\vec{M}_{|\Omega^+}\equiv\vec{M}_\infty$.
The convective flow is continuous through $\Gamma_\infty$ and tangential on $\Gamma$.
Hence $\rho$, $k$ and $\vec{M}$ are continuous through $\Gamma_\infty$, and $\vec{M}\cdot\vec{n}=0$ on $\Gamma$.

The physical quantities are associated with complex quantities with the following convention on, for instance, the acoustic pressure:
$p\leftrightarrow{\rm Re}\left(p \exp\left(-i\omega t\right)\right)$. In what follows, we always refer to the complex
quantity. Furthermore, the Hermitian product of two
vectors $\vec{U},\vec{W}\in\mathbb{C}^3$ is denoted by $\overline{\vec{U}}\cdot \vec{W}=\sum_{i=1}^3\overline{U_i}W_i$,
and the associated Euclidian norm in $\mathbb{C}^3$ is denoted by
$\|\cdot\|$.

\subsection{The convected Helmholtz equation}

In the interior domain $\Omega^-$, the convective flow is supposed to be stationary, inviscid, isentropic, potential and subsonic.
The acoustic effects are considered to be a first-order perturbation of this flow. With these assumptions, there exists an acoustic potential $\varphi$ such that $\vec{v}=\vec{\nabla}{\varphi}$.

Following~\cite[Equation (F27)]{introacous} and~\cite{Goldstein}, and making use of the acoustic potential,
the linearization of the Euler equations leads to
\begin{equation}
\rho\left(k^2\varphi+i k \vec{M}\cdot\vec{\nabla}{\varphi}\right)+\vec{\nabla}\cdot\left(\rho\left(\vec{\nabla}{\varphi}
-\left(\vec{M}\cdot\vec{\nabla}{\varphi}\right)\vec{M}+i k \varphi \vec{M}\right)\right)=g\textnormal{~~~in }\Omega,
\label{eq:pb_int}
\end{equation}
where $\varphi$ is the unknown acoustic potential, and $\rho$, $k$, $\vec{M}$, and $g$ are known.
Equation~\eqref{eq:pb_int} is the convected Helmholtz equation.
Under the assumption that the acoustic perturbations are perfectly reflected by the solid object,
the acoustic potential verifies an homogeneous Neumann boundary condition on $\Gamma$:
\begin{equation}
\label{eq:boundcond}
\vec\nabla\varphi\cdot\vec{n}=0\quad\textnormal{on }\Gamma.
\end{equation}
Problem~\eqref{eq:pb_int}-\eqref{eq:boundcond} is completed by a Sommerfeld-like boundary condition at infinity.

In the exterior domain $\Omega^+$ where the flow quantities are uniform, equation~\eqref{eq:pb_int} becomes
\begin{equation}
\label{eq:extpot}
\Delta\varphi+k_{\infty}^2\varphi+2ik_{\infty}\vec{M}_\infty\cdot\vec{\nabla}\varphi
-\vec{M}_\infty\cdot\vec{\nabla}\left(\vec{M}_\infty\cdot\vec{\nabla}\varphi \right)=g\textnormal{~~~in }\Omega^+.
\end{equation}
If there were no scattering object and if the convective flow
were uniform in $\mathbb{R}^3$ (and thus equal to the flow at infinity), the source term $g$
would create an acoustic potential denoted by $\varphi_{\rm inc}$ in $\mathbb{R}^3$.
This potential, which solves~\eqref{eq:extpot} in $\mathbb{R}^3$, has an analytical expression, and $\varphi_{\rm inc}$ and $\vec{n}\cdot\vec\nabla\varphi_{\rm inc}$ are continuous
across $\Gamma_\infty$.
The acoustic potential scattered by the solid object is defined as $\varphi_{\rm sc}=\varphi-\varphi_{\rm inc}$ in $\Omega^+$. 
Eliminating the known acoustic potential $\varphi_{\rm inc}$ created by the source yields
\begin{equation}
\label{eq:aeroacous_unif}
\Delta\varphi_{\rm sc}+k_{\infty}^2\varphi_{\rm sc}+2ik_{\infty}\vec{M}_\infty\cdot\vec{\nabla}\varphi_{\rm sc}
-\vec{M}_\infty\cdot\vec{\nabla}\left(\vec{M}_\infty\cdot\vec{\nabla}\varphi_{\rm sc} \right)=0\textnormal{~~~in }\Omega^+.
\end{equation}




\subsection{The Prandtl--Glauert transformation}
\label{sec:lorentz}

The Prandtl--Glauert transformation was introduced by Glauert~\cite{Glauert} to study the
compressible effects of the air on the lift of an airfoil and was applied to subsonic aeroacoustic problems by Amiet
and Sears~\cite{Amiet}.
Herein, the Prandtl--Glauert transformation is applied in the complete medium of propagation and is based on the reduced
velocity $\vec{M}_\infty$. This transformation consists in changing the space and time variables as
\begin{equation}
\left\{
 \begin{aligned}
\vec{x}'&=\gamma_\infty\left(\hat{\vec{M}}_\infty\cdot\vec{x}\right)\hat{\vec{M}}_\infty+\left(
\vec{x}-(\hat{\vec{M}}_\infty\cdot\vec{x})\hat{\vec{M}}_\infty\right)&\quad \vec{x}\in\Omega,\\
  t'&=t-\frac{\gamma_\infty^2}{c_\infty}\vec{M}_\infty\cdot\vec{x}&\quad t\in\mathbb{R},
 \end{aligned}
\right.
\end{equation}
where $\gamma_\infty=\frac{1}{\sqrt{1-M_{\infty}^2}}$ and $\hat{\vec{M}}_\infty=M_\infty^{-1}{\vec{M}_\infty}$ with
$M_\infty=|\vec{M}_\infty|$. The spatial transformation corresponds to a dilatation along $\hat{\vec{M}}_\infty$ of
magnitude $\gamma_\infty$, the component orthogonal to $\hat{\vec{M}}_\infty$ being unchanged.
In what follows, we suppose that the flow is subsonic
  everywhere, so that $M_\infty < 1$. Under this property,
the Prandtl--Glauert transformation is a $\mathcal{C}^{\infty}$-diffeomorphism from $\Omega\times\mathbb{R}$ to $\Omega'\times\mathbb{R}$,
where $\Omega'$ denotes the transformed medium of propagation.

\subsection{The transformed problem}
\label{sec:transformed_pb}


To apply the Prandtl--Glauert transformation to a PDE in the frequency domain, one has to change the differential operators as
\begin{equation}
\label{eq:precomp}
\vec\nabla u=\mathcal{N}\vec\nabla' u,\qquad
\vec\nabla \cdot\vec{U}=\vec\nabla' \cdot\mathcal{N}\vec{U},
\end{equation}
for a scalar-valued function $u$ and a vector-valued function $\vec{U}$. Here,
$\mathcal{N}=I+C_\infty\vec{M}_\infty\vec{M}^T_\infty$ with $C_\infty = \frac{\gamma_\infty-1}{M_{\infty}^2}$ and
$\vec{\nabla'}$ refers to derivatives with respect to the transformed variables $\vec{x}'$.
Moreover, it is readily verified that
\begin{equation}
\label{eq:precomp2}
\mathcal{N}\vec{M} =\vec{M}+C_\infty P\vec{M}_\infty,\qquad
\mathcal{N}\vec{M}_\infty =\gamma_\infty\vec{M}_\infty,\qquad
\mathcal{N}\vec{M}\cdot\vec{M}_\infty =\mathcal{N}\vec{M}_\infty\cdot\vec{M}=\gamma_\infty P,
\end{equation}
where $P = \vec{M}\cdot\vec{M}_\infty$.
Dividing equation~\eqref{eq:pb_int} by $\rho_\infty$ and applying~\eqref{eq:precomp} leads to
\begin{equation*}
rk^2\varphi+irk\vec{M}\cdot\mathcal{N}\vec\nabla'{\varphi}+\vec{\nabla}'\cdot\left(r\mathcal{N}\mathcal{N}\vec\nabla'{\varphi}\right)
-\vec{\nabla}'\cdot\left(r\left(\vec{M}\cdot\mathcal{N}\vec\nabla'{\varphi}\right)\mathcal{N}\vec{M}\right)+\vec{\nabla}'\cdot\left(irk \varphi \mathcal{N}\vec{M}\right)={\rho_\infty^{-1}}g,
\end{equation*}
where  $r = \frac{\rho}{\rho_\infty}$.
Let $f$ be such that $\varphi (\vec{x})=\alpha(\vec{x}')f(\vec{x}')$ with $\alpha(\vec{x}') = \exp\left(-i k_{\infty}\gamma_\infty\left(\vec{M}_\infty\cdot\vec{x}'\right)\right)$,
$\vec{x}'\in\Omega'$; $f_{\rm inc}$ and $f_{\rm sc}$ are defined from $\varphi_{\rm inc}$
and $\varphi_{\rm sc}$ in the same fashion, so that $f_{\rm inc}$ is analytically known, and defined in $\mathbb{R}^3$.
Expanding the derivatives with respect to $\alpha$
and using~\eqref{eq:precomp2} leads to
\begin{equation*}
\begin{aligned}
&\alpha rk^2 f+\alpha irk\vec{M}\cdot\mathcal{N}\vec\nabla'{f}+\alpha rk^2qPf+
\vec{\nabla}'\cdot\left(\alpha r\mathcal{N}\mathcal{N}\vec\nabla'{f}\right)
-\vec{\nabla}'\cdot\left(\alpha irk \gamma_\infty qf\vec{M}_\infty\right)\\
&\quad -\vec{\nabla}'\cdot\left(\alpha r\left(\vec{M}\cdot\mathcal{N}\vec\nabla'{f}\right)\mathcal{N}\vec{M}\right)
+\vec{\nabla}'\cdot\left(\alpha irk(1+qP)f\mathcal{N}\vec{M}\right)={\rho_\infty^{-1}}g,
\end{aligned}
\end{equation*}
where $q = \gamma_\infty^2\frac{k_\infty}{k}$ and where we used the fact that
$\vec\nabla'\alpha=-\alpha ik_\infty \gamma_\infty\vec{M_\infty}$.
Expanding again the derivatives with respect to $\alpha$
and using again~\eqref{eq:precomp2} as well as the symmetry of
$\mathcal{N}$, we infer after some calculations that
\begin{equation}
\label{eq:equa_init0}
r k^2 \beta f+i r k \vec{V}\cdot\vec{\nabla}'{f}
+\vec{\nabla}'\cdot\left(i r k f\vec{V} + r\Xi\vec{\nabla}'{f}\right)=\varsigma\quad\textnormal{in }{\Omega},
\end{equation}
where $\beta = \left(1+qP\right)^2-q^2M_\infty^2$,
$\vec{V} = \left(1+qP\right)\mathcal{N}\vec{M} -q\gamma_\infty\vec{M}_\infty$,
$\Xi = \mathcal{N}\mathcal{O}\mathcal{N}$ with $\mathcal{O} = I-\vec{M}\vec{M}^T$, and
$\varsigma(\vec{x}')={\rho_\infty^{-1}}\alpha^{-1}(\vec{x}')g(\vec{x}')$.

%
Consider now the boundary condition~\eqref{eq:boundcond}.
The normals on the initial geometry are denoted by $\vec{n}$, and the normals on the transformed geometry by $\vec{n}'$.
It is readily seen that
\begin{equation}
\label{eq:nprime_fonction_n}
\begin{aligned}
\vec{n}&=K_\infty\mathcal{N}\vec{n}',
\end{aligned}
\end{equation}
where $K_\infty$ is a normalization factor that is not needed in what follows.
Owing to~\eqref{eq:precomp} and~\eqref{eq:nprime_fonction_n},~\eqref{eq:boundcond} becomes
$\mathcal{N}\vec\nabla'\varphi\cdot\mathcal{N}\vec{n}'=0$. Hence,
$\mathcal{N}\vec\nabla'(\alpha f)\cdot\mathcal{N}\vec{n}'=0$, leading to
$\left(\mathcal{N}\vec\nabla'f-ik_\infty \gamma_\infty f\mathcal{N}\vec{M}_\infty\right)\cdot\mathcal{N}\vec{n}'=0$.
Since the flow is tangential on $\Gamma$, $\vec{M}\cdot\vec{n}=0$ on $\Gamma$. Hence, $\vec{M}\cdot\mathcal{N}\vec{n}'=0$
on $\Gamma'$, where $\Gamma'$ denotes the transformed boundary $\Gamma$, so that
\begin{equation}
\begin{aligned}
\label{eq:demo3}
\left(\mathcal{N}\vec\nabla'f-\left(\mathcal{N}\vec{M}\cdot\vec\nabla'f\right)\vec{M}
+ikf\left((1+qP)\vec{M}-\frac{k_\infty}{k}\gamma_\infty\mathcal{N}\vec{M}_\infty\right)\right)\cdot\mathcal{N}\vec{n}'=0.
\end{aligned}
\end{equation}
Using the symmetry of $\mathcal{N}$ and~\eqref{eq:precomp2},~\eqref{eq:demo3} leads to
\begin{equation}
\label{eq:bound}
\left(i r k f\vec{V}+r\Xi\vec{\nabla}'{f}\right)\cdot\vec{n}'=0\quad\textnormal{on }\Gamma'.
\end{equation}
Notice that although the term $\vec{M}\cdot\mathcal{N}\vec{n}'$
vanishes, it has been added to the expression of the boundary condition in order to obtain in~\eqref{eq:bound}
the normal component of the vector-valued function in
the divergence term of~\eqref{eq:equa_init0}. This point is crucial to derive coupled formulations.

In what follows, primes are omitted for brevity, and the transformed geometry, unknowns and operators are considered.
In summary, the transformed convected Helmholtz equation
together with the boundary condition and radiation condition at infinity
takes the form
\begin{subequations}
\label{eq:equa_init}
 \begin{align}
r k^2 \beta f+i r k \vec{V}\cdot\vec{\nabla}{f}
+\vec{\nabla}\cdot\left(i r k f\vec{V}+r\Xi\vec{\nabla}{f}\right)&=\varsigma\quad\textnormal{in }{\Omega},\label{eq:equa_init1}\\
\left(i r k f\vec{V}+r\Xi\vec{\nabla}{f}\right)\cdot\vec{n}&=0\quad\textnormal{on }\Gamma,\label{eq:equa_init2}\\
\lim_{r\rightarrow +\infty} r\left( \frac{\partial{(f-f_{\rm inc})}}{\partial{r}}-i \hat{k}_{\infty} (f-f_{\rm inc}) \right)&=0,\label{eq:equa_init3}
\end{align}
\end{subequations}
where $f$ is searched in $H^{1}_{\rm loc}(\Omega)$.
The Sommerfeld radiation condition~\eqref{eq:equa_init3} is written for the scattered potential, since some incident acoustic
potentials, e.g., plane waves, do not verify it.

\begin{proposition}
The matrix $\Xi$ is symmetric positive definite in $\Omega^-$ with
\begin{equation}
\label{eq:usefullcoercivity}
\overline{\vec{U}}\cdot\Xi
\vec{U}\geq\left(1-M_0^2\right)\left\|\vec{U}\right\|^2\quad \textnormal{ for all }\vec{U}\in\mathbb{C}^3,
\end{equation}
where $M_0$ is uniformly bounded away from $1$ since the convective flow is assumed to be subsonic.
Moreover, still in $\Omega^-$,
\begin{equation}
\overline{\vec{U}}\cdot\Xi\vec{W}\leq\frac{1+M_0^2}{1-M_\infty^2}\|\vec{U}\|\|\vec{W}\|\quad
\textnormal{ for all }\vec{U},\vec{W}\in\mathbb{C}^3.
\end{equation}
\end{proposition}

An important observation is that
in $\Omega^+$, $\beta=\gamma_\infty^2$, $\vec{V}=\vec{0}$ and $\Xi=I$, so that~\eqref{eq:equa_init1}
becomes
\begin{equation}
\label{eq:eqext}
\Delta f+\hat{k}_\infty^2 f=\varsigma\quad\textnormal{in }\Omega^+,
\end{equation}
where
\begin{equation}
\label{eq:kinfty}
\hat{k}_\infty=\gamma_\infty{k}_\infty. 
\end{equation}
Moreover, since ${\rm supp}(\varsigma)\subset\Omega^+$, $f_{\rm inc}$ satisfies
\begin{equation}
\label{eq:finc}
\Delta f_{\rm inc}+\hat{k}_\infty^2 f_{\rm inc}=\varsigma\quad\textnormal{in }\Omega^+,
\qquad\Delta f_{\rm inc}+\hat{k}_\infty^2 f_{\rm inc}=0\quad\textnormal{in }\mathbb{R}^3\backslash\overline{\Omega^+}.
\end{equation}
Eliminating $f_{\rm inc}$ in~\eqref{eq:eqext} yields
\begin{equation}
\Delta f_{\rm sc}+\hat{k}_\infty^2 f_{\rm sc}=0\quad\textnormal{in }\Omega^+.
\end{equation}
This is the classical Helmholtz equation with modified wave number $\hat{k}_\infty$.



\section{Coupling procedure}
\label{sec:coupling}

The purpose of this section is to derive two BEM-FEM coupled formulations for problem~\eqref{eq:equa_init}
using the tools presented in Section~\ref{sec:basic_ing}
and to analyze their well-posedness.

\subsection{The transmission problem}


We consider the following transmission problem:
\begin{subequations}
\label{eq:pb_sep}
 \begin{align}
r k^2 \beta f^-+i r k \vec{V}\cdot\vec{\nabla}{f^-}
+\vec{\nabla}\cdot\left(i r k f\vec{V}+r\Xi\vec{\nabla}{f}\right)^-&=0\quad\textnormal{in }{\Omega^-},\label{eq:pb_sep1}\\
\Delta f_{\rm sc} + \hat{k}_{\infty}^2 f_{\rm sc} &= 0\quad\textnormal{in }{\Omega^+},\label{eq:pb_sep2}\\
\left(\left(i r k {f}\vec{V}+r\Xi\vec{\nabla}{f}\right)\cdot\vec{n}\right)^-&=0\quad\textnormal{on }\Gamma,\label{eq:pb_sep3}\\
{\gamma_0^+}f^+-{\gamma_0^-}f^-&=0\quad\textnormal{on }\Gamma_\infty,\label{eq:pb_sep4}\\
{\gamma_1^+}f^+-{\gamma_1^-}f^-&=0\quad\textnormal{on }\Gamma_\infty,\label{eq:pb_sep5}\\
\lim_{r\rightarrow +\infty} r\left( \frac{\partial{(f^+-f^+_{\rm inc})}}{\partial{r}}-i \hat{k}_{\infty} (f^+-f^+_{\rm inc}) \right)&=0.\label{eq:pb_sep6}
 \end{align}
\end{subequations}
It is easily seen that Problem~\eqref{eq:equa_init} is equivalent to Problem~\eqref{eq:pb_sep}.

\begin{proposition}
\label{theouni}
Problem~\eqref{eq:pb_sep} has at most one solution in $H^{1}_{\rm loc}(\Omega)$.
\end{proposition}
We give the proof since the non-uniform convection coefficients
do not have enough regularity to apply the classical argument of analytical continuation.
\begin{proof}
Suppose $\varsigma=0$, so that $f_{\rm inc}=0$, and let $f\in H^1_{\rm loc}(\Omega)$ solve~\eqref{eq:pb_sep}.
Then, $f$ solves~\eqref{eq:equa_init}. Let $B$ be an open ball containing $\Omega^-$.
Let $f^t\in H(\Delta,B)$. Using Green's first identity,
\begin{equation}
\label{eq:demoo}
\begin{aligned}
0&=\int_{\Omega\cap B}\left(-rk^2\beta \overline{f}-ikr\vec{V}\cdot\vec\nabla \overline{f}-\vec\nabla\cdot\left(irk\overline{f}\vec{V} + r\Xi\vec\nabla \overline{f}\right)\right)f^t\\
&=\int_{\Omega\cap B}r\Xi\vec\nabla \overline{f}\cdot\vec\nabla f^t-rk^2\beta \overline{f}f^t-ikr\vec{V}\cdot(\vec\nabla \overline{f} f^t-\vec\nabla f^t \overline{f})-\left({\gamma_{1,\partial B}^-}f, \gamma_{0,\partial B}^-f^t\right)_{\partial B},
\end{aligned}
\end{equation}
where $\gamma_{0,\partial B}^-$ and $\gamma_{1,\partial B}^-$ are the Dirichlet and Neumann traces on $\partial B$ from $B$.
Taking $f^t=f$ yields
\begin{equation}
\begin{aligned}
\left({\gamma_{1,\partial B}^-}f, \gamma_{0,\partial B}^-f\right)_{\partial B}=\int_{\Omega\cap B}r\Xi\vec\nabla \overline{f}\cdot\vec\nabla f-rk^2\beta \overline{f}f-2kr\vec{V}\cdot\left({\rm Im}\vec\nabla \overline{f} f\right),
\end{aligned}
\end{equation}
so that ${\rm Im}\left({\gamma_{1,\partial B}^-}f, \gamma_{0,\partial B}^-f\right)_{\partial B}=0$. Using Rellich Lemma
(see~\cite[Lemma 9.9]{mclean}),  since $f\in H^1_{\rm loc}(\mathbb{R}^3\backslash\overline{B})$ solves the classical Helmholtz equation in $\mathbb{R}^3\backslash\overline{B}$ and satisfies the
Sommerfeld radiation condition, as well as ${\rm Im}\left({\gamma_{1,\partial B}^-}f, \gamma_{0,\partial B}^-f\right)_{\partial B}\geq 0$,
it is inferred that $f|_{\mathbb{R}^3\backslash\overline{B}}\equiv 0$. Equation~\eqref{eq:equa_init1} can be written
\begin{equation}
L(f):=\left(r k^2 \beta +\vec\nabla\cdot(irk\vec{V})\right)f+2i r k \vec{V}\cdot\vec{\nabla}{f}
+\vec{\nabla}\cdot\left(r\Xi\vec{\nabla}{f}\right)=0\quad\textnormal{in }\Omega.
\end{equation}
From~\cite[Theorem 1.1]{Garofalo}, since $r\Xi$ is uniformly elliptic with Lipschitz continuous coefficients, and $r k^2 \beta +\vec\nabla\cdot(irk\vec{V})$
and $2i r k \vec{V}$ have bounded coefficients, the differential operator $L$ satisfies
the strong unique continuation property in $\Omega$. Hence, $f|_{\mathbb{R}^3\backslash\overline{B}}\equiv 0$ implies
that $f|_{\Omega}\equiv 0$.
\end{proof}



\subsection{Weak formulation in the interior domain ${\Omega^-}$}

Let $\Phi=f|_{\Omega^-}$ where $f$ solves~\eqref{eq:pb_sep}. Multiplying~\eqref{eq:pb_sep1}
by a test function $\Phi^t\in H^1({\Omega^-})$ and using a Green formula together with the boundary condition~\eqref{eq:pb_sep3}
yields
\begin{equation}
\mathcal{V}(\Phi,\Phi^t)-\left({\gamma_{1}^-}{\Phi},{\gamma_{0}^-}{\Phi}^t\right)_{\Gamma_\infty}=0,
\end{equation}
with the sesquilinear form
\begin{equation}
\label{eq:defV}
\mathcal{V}(\Phi,\Phi^t)=\int_{{\Omega^-}}r\Xi\vec{\nabla}\overline{\Phi}\cdot\vec{\nabla}{{\Phi}^t}-\int_{{\Omega^-}}r k^2
\beta \overline{\Phi}{{\Phi}^t}+i \int_{{\Omega^-}}r k \vec{V}\cdot\left(\overline{\Phi}\vec{\nabla}{{\Phi}^t}-{{\Phi}^t}
\vec{\nabla}\overline{\Phi}\right).
\end{equation}
Using the transmission conditions~\eqref{eq:pb_sep4}-\eqref{eq:pb_sep5}, $\gamma_0^-\Phi={\gamma_0} f$ and
$\gamma_1^-\Phi={\gamma_1} f$, so that the coupling with the exterior problem can be written as $\gamma_1^-\Phi=DtN(\gamma_0^- \Phi)$.
This yields the following coupled formulation:
Find $\Phi\in H^1\left({\Omega^-}\right)$ such that $\forall \Phi^t\in H^1\left({\Omega^-}\right)$,
\begin{equation}
\label{eq:fv_int2}
\mathcal{V}(\Phi,\Phi^t)-\left(DtN({\gamma_0^-}\Phi),{\gamma_0^-}{\Phi^t}\right)_{\Gamma_\infty}=0.
\end{equation}



The coupling is carried out by taking as $DtN$ map in~\eqref{eq:fv_int2} the maps presented in Sections~\ref{sec:unstabdtn} and~\ref{sec:stabdtn}.
We recall that Helmholtz equations, as well as corresponding boundary integral operators, are written
on a geometry and for unknown functions that have been transformed by the Prandtl--Glauert transformation;
the wave number of the source is $\hat{k}_\infty=\gamma_\infty{k}_\infty$
(see~\eqref{eq:kinfty}).

\subsection{Unstable coupled formulation with one surface unknown}
\label{sec:unstabcoupl}

Injecting $DtN_{\rm unstab}(\gamma_0^-\Phi)$ from~\eqref{eq:DtNunst} into the formulation~\eqref{eq:fv_int2} yields,
using~\eqref{eq:deflam}, the following variational formulation:
Find $\left(\Phi,\lambda\right)\in \mathcal{H}$ such that, $\forall
\left(\Phi^t,\lambda^t\right)\in \mathcal{H}$,
\begin{subequations}
\label{eq:coupledvf}
 \begin{align}
\dsp \mathcal{V}(\Phi,\Phi^t)
+\left(N({\gamma_0^-}\Phi),{\gamma_0^-}\Phi^t\right)_{\Gamma_{\infty}}+\left(\left({\tilde{D}}-\frac{1}{2}I\right)(\lambda),{\gamma_0^-}\Phi^t\right)_{\Gamma_{\infty}}
&= \left({\gamma_1}f_{\rm inc},{\gamma_0^-}\Phi^t\right)_{\Gamma_{\infty}},\\
\dsp \left(\left({D}-\frac{1}{2}I\right)({\gamma_0^-}\Phi),\lambda^t\right)_{\Gamma_{\infty}}-\left(S(\lambda),\lambda^t\right)_{\Gamma_{\infty}}
&= -\left({\gamma_0}f_{\rm inc},\lambda^t\right)_{\Gamma_{\infty}},
\end{align}
\end{subequations}
with product space $\mathcal{H}=H^{1}\left({\Omega^-}\right)\times H^{-\frac{1}{2}}\left({\Gamma_\infty}\right)$
and inner product $\left(\left(\Phi,\lambda\right),\left(\Phi^t,\lambda^t\right)\right)_{\mathcal{H}}=\left(\Phi,
\Phi^t\right)_{H^{1}\left({\Omega^-}\right)}+\left(\lambda,\lambda^t\right)_{H^{-\frac{1}{2}}\left({\Gamma_\infty}\right)}$.

If $f$ solves~\eqref{eq:pb_sep}, then $(f^-,\gamma_1 f)$ solves~\eqref{eq:coupledvf}, and more generally, 
$(f^-,\gamma_1 f+\lambda^*)$ solves~\eqref{eq:coupledvf} for all $\lambda^*\in\ker(S)$. Hence,
in the case of resonant frequencies,
i.e., $-\hat{k}_\infty^2\in\Lambda$, where $\ker(S)$ is not trivial,~\eqref{eq:coupledvf} admits infinitely many solutions.
Conversely,
if $(\Phi,\lambda)$ solves~\eqref{eq:coupledvf}, then $\mathcal{R}(\Phi,\lambda)$ solves~\eqref{eq:pb_sep}, where
$\mathcal{R}:\mathcal{H}\rightarrow H^1_{\rm loc}(\Omega\backslash\Gamma_\infty)$ is such that
$\mathcal{R}({\Phi},{\lambda})|_{\Omega^-}={\Phi}$ and
$\mathcal{R}({\Phi},{\lambda})|_{\Omega^+}=(-\mathcal{S}({\lambda}) + \mathcal{D}(\gamma_0^-{\Phi})+f_{\rm inc})|_{\Omega^+}$.
Notice that $\mathcal{R}(0,\lambda^*)=0$ for all $\lambda^*\in\ker(S)$. Thus, in the case of resonant frequencies,
where~\eqref{eq:coupledvf} admits infinitely many solutions, all of these solutions produce the same solution of~\eqref{eq:pb_sep}.
However, we will see in Section~\ref{sec:numres} that
the numerical procedure to approximate~\eqref{eq:coupledvf} becomes ill-conditioned so that
$\mathcal{R}({\Phi},{\lambda})$ is dominated by numerical
errors. For this reason, the formulation~\eqref{eq:coupledvf} is called unstable.

\begin{theorem}
\label{theounstab}
If $-\hat{k}_\infty^2\notin\Lambda$, problem~\eqref{eq:coupledvf} is well-posed.
\end{theorem}
\begin{proof}
The proof is omitted since it proceeds in the same fashion as that of Theorem~\ref{prop_exi_uni_stab} below.
\end{proof}

\begin{remark}[Symmetry of the system]
In the system~\eqref{eq:coupledvf}, the only non-symmetric contribution results from the vector $\vec{V}$
in the sesquilinear form $\mathcal{V}$, cf.~\eqref{eq:defV}.
The system becomes symmetric when the flow is uniform everywhere.
Notice that this is not an Hermitian symmetry.
The operators $D$ and $\tilde{D}$ are dual but not adjoint.
\end{remark}

\subsection{Stable coupled formulation with two surface unknowns}
\label{sec:stabcoupl}

Injecting $DtN_{\rm stab}(\gamma_0^-\Phi)$ from~\eqref{eq:dtnst} into the formulation~\eqref{eq:fv_int2} yields,
using~\eqref{eq:lambdast} and~\eqref{eq:pst}, the following variational formulation:
Find $\left(\Phi,\lambda,p\right)\in \mathbb{H}$ such that
$\forall\left(\Phi^t,\lambda^t, p^t\right)\in \mathbb{H}$,
\begin{subequations}
\label{eq:weakcouptrans}
\begin{align}
\dsp \mathcal{V}(\Phi,\Phi^t)+\left(N({\gamma_0^-}\Phi),{\gamma_0^-}\Phi^t\right)_{\Gamma_{\infty}}+\left(\left({\tilde{D}}-\frac{1}{2}I\right)(\lambda),{\gamma_0^-}\Phi^t\right)_{\Gamma_{\infty}}
&= \left({\gamma_1}f_{\rm inc},{\gamma_0^-}\Phi^t\right)_{\Gamma_{\infty}},\label{eq:weakcouptrans1}\\
\dsp \left(\left({D}-\frac{1}{2}I\right)({\gamma_0^-}\Phi),\lambda^t\right)_{\Gamma_{\infty}}-\left(S(\lambda),\lambda^t\right)_{\Gamma_{\infty}}
+i\overline{\eta}\left(p,\lambda^t\right)_{\Gamma_{\infty}}
&= -\left({\gamma_0}f_{\rm inc},\lambda^t\right)_{\Gamma_{\infty}},\label{eq:weakcouptrans2}\\
\dsp \left(N({\gamma_0^-}\Phi),p^t\right)_{\Gamma_{\infty}}+\left(\left({\tilde{D}}+\frac{1}{2}I\right)(\lambda),p^t\right)_{\Gamma_{\infty}}
-\delta_{\Gamma_\infty}(p,p^t)&= \left({\gamma_1}f_{\rm inc},p^t\right)_{\Gamma_\infty},\label{eq:weakcouptrans3}
\end{align}
\end{subequations}
with product space $\mathbb{H}=H^{1}\left({\Omega^-}\right)\times H^{-\frac{1}{2}}\left({\Gamma_\infty}\right)\times H^{1}({\Gamma_\infty})$
and inner product
$\left(\left(\Phi,\lambda,p\right),\left(\Phi^t,\lambda^t,p^t\right)\right)_{\mathbb{H}}=\left(\Phi, \Phi^t\right)_{H^{1}\left({\Omega^-}\right)}+\left(\lambda,\lambda^t\right)_{H^{-\frac{1}{2}}\left({\Gamma_\infty}\right)}
+\left(p,p^t\right)_{H^{1}\left({\Gamma_\infty}\right)}$.

If $f$ solves~\eqref{eq:pb_sep}, then $(f^-,\gamma_1 f,0)$ solves~\eqref{eq:weakcouptrans}. Conversely,
if $(\Phi,\lambda, p)$ solves~\eqref{eq:weakcouptrans}, then $\mathcal{R}(\Phi,\lambda)$ solves~\eqref{eq:pb_sep}
and $p=0$, where $\mathcal{R}$ is defined in Section~\ref{sec:unstabcoupl}.

\begin{theorem}
\label{prop_exi_uni_stab}
Problem~\eqref{eq:weakcouptrans} is well-posed at all frequencies.
\end{theorem}
\begin{proof}
First, using Proposition~\ref{theouni}, we can show that Problem~\eqref{eq:weakcouptrans} has at most one solution.
Then, consider the two sesquilinear forms $a_1$ and $a_2$ on $\mathbb{H}\times\mathbb{H}$ such that
\begin{equation}
\begin{array}{ll}
\dsp a_1\left(\left(\Phi,\lambda,p\right),\left(\Phi^t,\lambda^t,p^t\right)\right)=\int_{{\Omega^-}}{r\Xi\vec{\nabla}{\overline{\Phi}}\cdot\vec{\nabla}{{\Phi}^t}}
+\left(N^0({\gamma_0^-}\Phi),{\gamma_0^-}\Phi^t\right)_{\Gamma_{\infty}}+{\left(S^0(\lambda),\lambda^t\right)_{\Gamma_{\infty}}}+\delta_{\Gamma_\infty}(p,p^t)\\
\vspace{0.3cm}
\qquad \dsp +\left(\left(\tilde{D}^0-\frac{1}{2}I\right)(\lambda),{\gamma_0^-}\Phi^t\right)_{\Gamma_{\infty}}-{\left(\left({D^0}-\frac{1}{2}I\right)({\gamma_0^-}\Phi),{\lambda}^t\right)_{\Gamma_{\infty}}},\\
\dsp a_2\left(\left(\Phi,\lambda,p\right),\left(\Phi^t,\lambda^t,p^t\right)\right)=-\int_{{\Omega^-}}{{rk^2}\beta \overline{\Phi}{\Phi}^t}
+i\int_{{\Omega^-}}{{r k} \vec{V}\cdot\left({\overline{\Phi}}\vec{\nabla}{{\Phi}^t}-{{\Phi}^t}\vec{\nabla}{\overline{\Phi}}\right)}
+\left(\left(N-N^0\right)({\gamma_0^-}\Phi),{\gamma_0^-}\Phi^t\right)_{\Gamma_{\infty}}\\
\vspace{0.3cm}
\dsp \qquad +{\left(\left(S-S^0,\lambda^t\right)(\lambda)\right)_{\Gamma_{\infty}}}+\left(\left(\tilde{D}-\tilde{D}^0\right)(\lambda),{\gamma_0^-}\Phi^t\right)_{\Gamma_{\infty}}
-{\left(\left({D}-{D^0},{\lambda}^t\right)({\gamma_0^-}\Phi)\right)_{\Gamma_{\infty}}}-i\overline{\eta}{\left(p,\lambda^t\right)_{\Gamma_{\infty}}}\\
\vspace{0.3cm}
\dsp \qquad -\left(N({\gamma_0^-}\Phi),p^t\right)_{\Gamma_{\infty}}-\left(\left({\tilde{D}}+\frac{1}{2}I\right)(\lambda),p^t\right)_{\Gamma_{\infty}},
\end{array}
\end{equation}
where $S^0$, $D^0$, $\tilde{D}^0$ and $N^0$ are the boundary integral operators $S$, $D$, $\tilde{D}$ and $N$ for
$\hat{k}_\infty=0$.
For the volumic term $\mathcal{V}$, our writing of the convected Helmholtz equation using the
Prandtl--Glauert transformation in the form~\eqref{eq:equa_init1} allows us to readily see that $a_1$ is $\mathbb{H}\text{-coercive}$
owing to~\eqref{eq:usefullcoercivity}. Then, since the linear map associated with $a_2$ is classically compact from
$\mathbb{H}$ into $\mathbb{H}$ (see~\cite[Lemma 3.9.8]{sauter}), the assertion follows by the Fredholm alternative and the uniqueness of the solution.
\end{proof}


\section{Finite-dimensional approximation}
\label{sec:findim}

The coupled formulations~\eqref{eq:coupledvf} and~\eqref{eq:weakcouptrans} are approximated by FEM and BEM.
We briefly recall the underlying results from both theories. Then, we discuss in more detail the structure of the
linear systems and the algorithms for their numerical resolution.

\subsection{Discrete finite element spaces}


Let $\mathcal{M}$ be a shape-regular tetrahedral mesh of ${\Omega^-}$. The mesh $\mathcal{F}_\infty$ of $\Gamma_\infty$
is composed of the boundary faces of $\mathcal{M}$. Let $h_\mathcal{M}>0$ denote the mesh size,
$V^1_\mathcal{M}$ the space of continuous piecewise affine polynomials on $\mathcal{M}$,
$S^0_\mathcal{M}$ the space of piecewise constant polynomials on $\mathcal{F}_\infty$, and
$S^1_\mathcal{M}$ the space of continuous piecewise affine polynomials on $\mathcal{F}_\infty$.
Let $\mathcal{H}_\mathcal{M}=V^1_\mathcal{M}\times S^0_\mathcal{M}$, and $\mathbb{H}_\mathcal{M}=V^1_\mathcal{M}\times S^0_\mathcal{M}\times S^1_\mathcal{M}$.
The discretization of~\eqref{eq:coupledvf} reads: Find $\left(\Phi_\mathcal{M},\lambda_\mathcal{M}\right)\in \mathcal{H}_\mathcal{M}$
such that, $\forall \left(\Phi_\mathcal{M}^t,\lambda_\mathcal{M}^t\right)\in \mathcal{H}_\mathcal{M}$,
\begin{equation}
a^{\rm unstab}\left(\left(\Phi_\mathcal{M},\lambda_\mathcal{M}\right), \left(\Phi_\mathcal{M}^t,\lambda_\mathcal{M}^t\right)\right)=b^{\rm unstab}\left(\Phi_\mathcal{M}^t,\lambda_\mathcal{M}^t\right),
\label{eq:coupledvfun_num}
\end{equation}
with $a^{\rm unstab}$ and $b^{\rm unstab}$ readily deduced from~\eqref{eq:coupledvf},
while the discretization of~\eqref{eq:weakcouptrans} reads:
Find $\left(\Phi_\mathcal{M},\lambda_\mathcal{M}, p_\mathcal{M}\right)\in \mathbb{H}_\mathcal{M}$
such that, $\forall \left(\Phi_\mathcal{M}^t,\lambda_\mathcal{M}^t, p_\mathcal{M}^t\right)\in \mathbb{H}_\mathcal{M}$,
\begin{equation}
a^{\rm stab}\left(\left(\Phi_\mathcal{M},\lambda_\mathcal{M}, p_\mathcal{M}\right), \left(\Phi_\mathcal{M}^t,\lambda_\mathcal{M}^t, p_\mathcal{M}^t\right)\right)=b^{\rm stab}\left(\Phi_\mathcal{M}^t,\lambda_\mathcal{M}^t, p_\mathcal{M}^t\right),
\label{eq:coupledvfst_num}
\end{equation}
with $a^{\rm stab}$ and $b^{\rm stab}$ readily deduced from~\eqref{eq:weakcouptrans}.
Since $\mathcal{H}_\mathcal{M}\subset \mathcal{H}$ and $\mathbb{H}_\mathcal{M}\subset \mathbb{H}$, both approximations are conforming.

In what follows, $A\lesssim B$ denotes the inequality $A\leq cB$ with positive constant $c$ independent
of the mesh size and of the discrete and exact solutions.
Owing to classical approximation properties~\cite{Brenner, Ern, sauter}, there holds,
$\forall (\Phi,\lambda)\in H^2(\Omega^-)\times H^{\frac{1}{2}}(\Gamma_\infty)$,
\begin{equation}
\label{eq:app_con}
\inf_{\left(\Phi_\mathcal{M},\lambda_\mathcal{M}\right)\in\mathcal{H}_\mathcal{M}}\left\|(\Phi,\lambda)-(\Phi_\mathcal{M},\lambda_\mathcal{M})\right\|_{\mathcal{H}}\lesssim h_\mathcal{M}\left(\left\|\Phi\right\|_{H^2(\Omega^-)}+\left\|\lambda\right\|_{H^{\frac{1}{2}}\left(\Gamma_\infty\right)}\right),
\end{equation}
and $\forall (\Phi, \lambda, p)\in H^2(\Omega^-)\times H^{\frac{1}{2}}(\Gamma_\infty)\times H^2(\Gamma_\infty)$,
\begin{equation}
\label{eq:app_const}
\inf_{\left(\Phi_\mathcal{M},\lambda_\mathcal{M}, p_\mathcal{M}\right)\in\mathbb{H}_\mathcal{M}}\left\|(\Phi,\lambda,p)-(\Phi_\mathcal{M},\lambda_\mathcal{M},p_\mathcal{M})\right\|_{\mathbb{H}}\lesssim h_\mathcal{M}\left(\left\|\Phi\right\|_{H^2(\Omega^-)}+\left\|\lambda\right\|_{H^{\frac{1}{2}}\left(\Gamma_\infty\right)}+\left\|p\right\|_{H^2(\Gamma_\infty)}\right).
\end{equation}

\begin{remark}
Taking a polynomial approximation with one order less for $H^{-\frac{1}{2}}\left(\Gamma_\infty\right)$ than for $H^{1}\left({\Omega^-}\right)$ and $H^{1}\left(\Gamma_\infty\right)$
yields that all the approximations have the same order in $h_\mathcal{M}$.
\end{remark}

The following error estimates follow from~\cite[Theorem 13]{Wendland}:
If $-\hat{k}_\infty^2\notin \Lambda$ and $h_\mathcal{M}$ is small enough, the discrete problem~\eqref{eq:coupledvfun_num} has a unique solution
$\left(\Phi_\mathcal{M},\lambda_\mathcal{M}\right)\in\mathcal{H}_{\mathcal{M}}$, and the following optimal error estimate holds:
\begin{equation}
\label{eq:prop_cvg}
 \|\left(\Phi,\lambda\right)-\left(\Phi_\mathcal{M},\lambda_\mathcal{M}\right)\|_{\mathcal{H}}\lesssim \inf_{\left(\Phi_\mathcal{M}^t,\lambda_\mathcal{M}^t\right)\in\mathcal{H}_\mathcal{M}} \|\left(\Phi,\lambda\right)-\left(\Phi_\mathcal{M}^t,\lambda_\mathcal{M}^t\right)\|_{\mathcal{H}},
\end{equation}
where $\left(\Phi,\lambda\right)$ is the unique solution of~\eqref{eq:coupledvf}.
Moreover, at all frequencies and if $h_\mathcal{M}$ is small enough, the discrete problem~\eqref{eq:coupledvfst_num} has a unique solution
$\left(\Phi_\mathcal{M},\lambda_\mathcal{M}, p_\mathcal{M}\right)\in\mathbb{H}_{\mathcal{M}}$, and the following optimal error estimate holds:
\begin{equation}
\label{eq:prop_cvgst}
 \|\left(\Phi,\lambda,p\right)-\left(\Phi_\mathcal{M},\lambda_\mathcal{M}, p_\mathcal{M}\right)\|_{\mathbb{H}}\lesssim \inf_{\left(\Phi_\mathcal{M}^t,\lambda_\mathcal{M}^t,p_\mathcal{M}^t\right)\in\mathbb{H}_\mathcal{M}} \|\left(\Phi,\lambda,p\right)-\left(\Phi_\mathcal{M}^t,\lambda_\mathcal{M}^t,p_\mathcal{M}^t\right)\|_{\mathbb{H}},
\end{equation}
where $\left(\Phi,\lambda,p\right)$ is the unique solution of~\eqref{eq:weakcouptrans}.

\begin{remark}
The constant in~\eqref{eq:prop_cvg} depends on $\hat{k}_\infty$, and its value explodes as $-\hat{k}_\infty^2$ tends to
an element of $\Lambda$. The constant in~\eqref{eq:prop_cvgst} depends on $\hat{k}_\infty$ as well, but remains bounded on
any bounded set of frequencies.
\end{remark}

\subsection{Structure of linear systems}
\label{sec:discrete}

\subsubsection{Unstable formulation with one surface unknown}
Let $(\theta_i)_{1\leq i\leq p}$ and $(\psi_i)_{1\leq i\leq q}$ denote finite element bases for $V^1_{\mathcal{M}}$ and $S^0_\mathcal{M}$ respectively.
These basis functions are real-valued.
The decompositions of $\Phi_{\mathcal{M}}\in V^1_{\mathcal{M}}$ and $\lambda_{\mathcal{M}}\in S^0_{\mathcal{M}}$ on these bases
are written in the form $\Phi_{\mathcal{M}}=\sum_{i=1}^p {\Phi_{\mathcal{M}}}_i \theta_i$ and $\lambda_{\mathcal{M}}=\sum_{i=1}^q {\lambda_{\mathcal{M}}}_i \psi_i$.
Let
\begin{equation}
u^{\rm unstab}_{\mathcal{M}}=\begin{pmatrix}
\begin{array}{c}
({\Phi_{\mathcal{M}}}_i)_{~{1\leq i\leq p}} \\
({\lambda_{\mathcal{M}}}_i)_{~{1\leq i\leq q}}
\end{array}
\end{pmatrix},\qquad
B^{\rm unstab}=\begin{pmatrix}
\begin{array}{c}
\left({\gamma_1}f_{\rm inc}, \gamma_0^-\theta_i\right)_{{\Gamma_\infty}~{1\leq i\leq p}}\\
- \left({\gamma_0}f_{\rm inc},\psi_i\right)_{{\Gamma_\infty}~{1\leq i\leq q}}
\end{array}
\end{pmatrix},
\end{equation}
\begin{equation}
A^{\rm unstab}=\begin{pmatrix}
\begin{array}{c|c}
\mathcal{V}(\theta_j,\theta_i)+{\left(N(\gamma_0^-\theta_j),\gamma_0^-\theta_i\right)_{\Gamma_{\infty}}} &
{\left(\left({\tilde{D}}-\frac{1}{2}I\right)(\psi_j),\gamma_0^-\theta_i\right)_{\Gamma_{\infty}}} \\
\hline
{\left(\left({D}-\frac{1}{2}I\right)(\gamma_0^-\theta_j),\psi_i\right)_{\Gamma_{\infty}}}&
-{\left(S(\psi_j),\psi_i\right)_{\Gamma_{\infty}}}
\end{array}
\end{pmatrix},
\end{equation}
where in $A^{\rm unstab}$ the index $i$ refers to the rows and the index $j$ to the columns.
The linear system resulting from~\eqref{eq:coupledvfun_num} is
\begin{equation}
\label{eq:mat}
A^{\rm unstab}u^{\rm unstab}_\mathcal{M}=B^{\rm unstab}.
\end{equation}
To better understand the structure of the linear system~\eqref{eq:mat}, the basis functions
$(\theta_i)_{1\leq i\leq p}$ of $V^1_{\mathcal{M}}$ are separated into two sets: the basis functions
$(\theta^{\mathcal{F}_\infty}_i)_{1\leq i\leq p^{\mathcal{F}_\infty}}$
associated with the mesh vertices
located in ${\mathcal{F}_\infty}$, and
$(\theta^{\mathring{\mathcal{M}}}_i)_{1\leq i\leq
  p^{\mathring{\mathcal{M}}}}$ associated with the mesh
  vertices located in $\Omega^-$; clearly, $p=p^{\mathcal{F}_\infty}+p^{\mathring{\mathcal{M}}}$.
The matrix $A^{\rm unstab}$ can then be further decomposed as
\begin{equation}
\label{eq:matrix:unst:blocks}
A^{\rm unstab}=\begin{pmatrix}
\begin{array}{c|c|c}
A^{\rm unstab}_{1,1} &
A^{\rm unstab}_{1,2} &
0\\
\hline
A^{\rm unstab}_{2,1} &
A^{\rm unstab}_{2,2} &
A^{\rm unstab}_{2,3}\\
\hline
0 &
A^{\rm unstab}_{3,2} &
A^{\rm unstab}_{3,3}
\end{array}
\end{pmatrix},
\end{equation}
where $\left(A^{\rm unstab}_{1,1}\right)_{i,j}=\mathcal{V}(\theta^{\mathring{\mathcal{M}}}_j,\theta^{\mathring{\mathcal{M}}}_i)$,
$\left(A^{\rm unstab}_{1,2}\right)_{i,j}=\mathcal{V}(\theta^{\mathcal{F}_\infty}_j,\theta^{\mathring{\mathcal{M}}}_i)$,
$\left(A^{\rm unstab}_{2,1}\right)_{i,j}=\mathcal{V}(\theta^{\mathring{\mathcal{M}}}_j,\theta^{\mathcal{F}_\infty}_i)$,
$\left(A^{\rm unstab}_{2,2}\right)_{i,j}=\mathcal{V}(\theta^{\mathcal{F}_\infty}_j,\theta^{\mathcal{F}_\infty}_i) +{\left(N(\gamma_0^-\theta^{\mathcal{F}_\infty}_j),\gamma_0^-\theta^{\mathcal{F}_\infty}_i\right)_{\Gamma_{\infty}}}$,
$\left(A^{\rm unstab}_{2,3}\right)_{i,j}={\left(\left({\tilde{D}}-\frac{1}{2}I\right)(\psi_j),\gamma_0^-\theta^{\mathcal{F}_\infty}_i\right)_{\Gamma_{\infty}}}$,
$\left(A^{\rm unstab}_{3,2}\right)_{i,j}={\left(\left({D}-\frac{1}{2}I\right)(\gamma_0^-\theta^{\mathcal{F}_\infty}_j),\psi_i\right)_{\Gamma_{\infty}}}$ and
$\left(A^{\rm unstab}_{3,3}\right)_{i,j}=-{\left(S(\psi_j),\psi_i\right)_{\Gamma_{\infty}}}$
All the blocks are complex-valued.
The blocks $A^{\rm unstab}_{1,1}$, $A^{\rm unstab}_{1,2}$ and $A^{\rm unstab}_{2,1}$ are sparse.
The block $A^{\rm unstab}_{1,1}$ is not symmetric, and the block $A^{\rm unstab}_{1,2}$
is neither the transpose nor the Hermitian transpose of the block $A^{\rm unstab}_{2,1}$.
The block $A^{\rm unstab}_{2,2}$ has two contributions: one sparse and nonsymmetric and one dense and symmetric;
therefore, this block is dense and nonsymmetric. The blocks $A^{\rm unstab}_{2,3}$, $A^{\rm unstab}_{3,2}$ and $A^{\rm unstab}_{3,3}$
are dense. The block $A^{\rm unstab}_{2,3}$ is the transpose of the block $A^{\rm unstab}_{3,2}$, and the block $A^{\rm unstab}_{3,3}$ is symmetric.


\subsubsection{Stable formulation with two surface unknowns}

Let $(\xi_i)_{1\leq i\leq r}$ denote a finite element basis for $S^1_\mathcal{M}$.
The decomposition of $p_{\mathcal{M}}\in S^1_{\mathcal{M}}$ on this basis
is written in the form $p_{\mathcal{M}}=\sum_{i=1}^r {p_{\mathcal{M}}}_i \xi_i$. Let
\begin{equation}
u^{\rm stab}_\mathcal{M}=\begin{pmatrix}
\begin{array}{c}
({\Phi_{\mathcal{M}}}_i)_{~{1\leq i\leq p}} \\
({\lambda_{\mathcal{M}}}_i)_{~{1\leq i\leq q}}\\
({p_{\mathcal{M}}}_i)_{~{1\leq i\leq r}}
\end{array}
\end{pmatrix},\qquad
B^{\rm stab}=\begin{pmatrix}
\begin{array}{c}
 \left({\gamma_1}f_{\rm inc}, \gamma_0^-\theta_i\right)_{{\Gamma_\infty}~{1\leq i\leq p}}\\
-\left({\gamma_0}f_{\rm inc},\psi_i\right)_{{\Gamma_\infty}~{1\leq i\leq q}}\\
 \left({\gamma_1}f_{\rm inc}, \xi_i\right)_{{\Gamma_\infty}~{1\leq i\leq r}}
\end{array}
\end{pmatrix},
\end{equation}
\begin{equation}
A^{\rm stab}=\begin{pmatrix}
\begin{array}{c|c|c}
\mathcal{V}(\theta_j,\theta_i)+{\left(N(\gamma_0^-\theta_j),\gamma_0^-\theta_i\right)_{\Gamma_{\infty}}} &
{\left(\left({\tilde{D}}-\frac{1}{2}I\right)(\psi_j),\gamma_0^-\theta_i\right)_{\Gamma_{\infty}}} &
0 \\
\hline
{\left(\left({D}-\frac{1}{2}I\right)(\gamma_0^-\theta_j),\psi_i\right)_{\Gamma_{\infty}}} &
-{\left(S(\psi_j),\psi_i\right)_{\Gamma_{\infty}}} &
i\overline{\eta}{\left(\xi_j,\psi_i\right)_{\Gamma_{\infty}}}\\
\hline
{\left(N(\gamma_0^-\theta_j),\xi_i\right)_{\Gamma_{\infty}}} &
{\left(\left({\tilde{D}}-\frac{1}{2}I\right)(\psi_j),\xi_i\right)_{\Gamma_{\infty}}}&
-\delta_{\Gamma_\infty}{\left(\xi_j,\xi_i\right)_{\Gamma_{\infty}}}\\
\end{array}
\end{pmatrix},
\end{equation}
with the same convention as above on the indices $i$ and $j$ of $A^{\rm stab}$.
The linear system resulting from~\eqref{eq:coupledvfst_num} is
\begin{equation}
\label{eq:matstab}
A^{\rm stab}u^{\rm stab}_\mathcal{M}=B^{\rm stab}.
\end{equation}

As in the previous section, the matrix of the linear
  system~\eqref{eq:matstab} is further decomposed as
\begin{equation}
\label{eq:matrix:st:blocks}
A^{\rm stab}=\begin{pmatrix}
\begin{array}{c|c|c|c}
A^{\rm stab}_{1,1} &
A^{\rm stab}_{1,2} &
0 &
0
\\
\hline
A^{\rm stab}_{2,1} &
A^{\rm stab}_{2,2} &
A^{\rm stab}_{2,3}&
0\\
\hline
0 &
A^{\rm stab}_{3,2} &
A^{\rm stab}_{3,3}&
A^{\rm stab}_{3,4}\\
\hline
0 &
A^{\rm stab}_{4,2} &
A^{\rm stab}_{4,3}&
A^{\rm stab}_{4,4}
\end{array}
\end{pmatrix},
\end{equation}
where $A^{\rm stab}_{1,1}$, $A^{\rm stab}_{1,2}$, $A^{\rm stab}_{1,3}$, $A^{\rm stab}_{2,1}$, $A^{\rm stab}_{2,2}$,
$A^{\rm stab}_{2,3}$, $A^{\rm stab}_{3,2}$ and $A^{\rm stab}_{3,3}$ are the same as their corresponding counterparts in~\eqref{eq:matrix:unst:blocks},
and $\left(A^{\rm stab}_{3,4}\right)_{i,j}=i\overline{\eta}{\left(\xi_j,\psi_i\right)_{\Gamma_{\infty}}}$,
$\left(A^{\rm stab}_{4,2}\right)_{i,j}={\left(N(\gamma_0^-\theta^{\mathcal{F}_\infty}_j),\xi_i\right)_{\Gamma_{\infty}}}$,
$\left(A^{\rm stab}_{4,3}\right)_{i,j}={\left(\left({\tilde{D}}-\frac{1}{2}I\right)(\psi_j),\xi_i\right)_{\Gamma_{\infty}}}$ and
$\left(A^{\rm stab}_{4,4}\right)_{i,j}=-\delta_{\Gamma_\infty}{\left(\xi_j,\xi_i\right)_{\Gamma_{\infty}}}$.
All the blocks are complex-valued.
The blocks $A^{\rm stab}_{3,3}$ and $A^{\rm stab}_{4,4}$ are sparse, whereas the
blocks $A^{\rm stab}_{4,2}$ and $A^{\rm stab}_{4,3}$ are dense. The block $A^{\rm stab}_{4,4}$ is symmetric.

\subsection{Numerical resolution}

Both the unstable and stable formulations have been implemented in the EADS in-house boundary element software called ACTIPOLE.
This software can treat general three-dimensional geometries. The iterative solver
is a GMRES solver~\cite{GMRES} with no restart,
suitable for non-symmetric linear systems (an additional
  feature of the solver is that it can treat multiple right-hand sides~\cite{Langou, blockgcr}).
The specificity of each block in~\eqref{eq:matrix:unst:blocks} or~\eqref{eq:matrix:st:blocks} is taken into account.
Matrix-vector products involving sparse blocks are
optimized accordingly, and those involving dense blocks resulting from
boundary integral terms can be accelerated using a fast multipole method
and out-of-core parallelization techniques.

The preconditioner uses a combination of a sparse approximate inverse (SPAI) preconditioner~\cite{Carpentieri, precomp}
and the sparse direct solver MUMPS~\cite{MUMPS}.
More precisely, for the dense diagonal blocks $A^{\rm unstab}_{2,2}$, $A^{\rm unstab}_{3,3}$ and $A^{\rm stab}_{2,2}$,
$A^{\rm stab}_{3,3}$, the SPAI preconditioner searches for an
approximation of the inverse of these blocks. Letting $A$ denote any of
these blocks, 
$A$ is made sparse by keeping, in each column, the interaction terms between the corresponding basis function and
the ones in its vicinity (in the sense of vertices or faces). The result of this operation is denoted
by $A^{\rm sp}$, and we define the set of matrices having the same
sparsity pattern, i.e., $\mathcal{S}_{A^{\rm sp}}=\{M\in\mathbb{C}^{n,n}~|~M_{i,j}=0 \,~\forall~1\leq~i,j\leq~n \mbox{ s.t. }
A^{\rm sp}_{i,j}~=~0\}$ where $n$ is the number of rows of $A$.
The SPAI preconditioner of $A$ is then given by $P=\underset{M\in \mathcal{S}_{A^{\rm sp}}}{\rm argmin}{\|A^{\rm sp}M-I
\|_{F}}$, where $\|\cdot\|_F$ denotes the Frobenius norm.
For the blocks $A^{\rm unstab}_{2,2}$ and $A^{\rm stab}_{2,2}$, the SPAI preconditioner is computed ignoring
the volumic contributions. For the sparse diagonal blocks, the preconditioner is taken as the inverse of each block.
The inverse is not actually computed: since MUMPS provides a factorization of each of these blocks, each time a
preconditioner-vector product is needed when constructing the Krylov vectors of the iterative method, two triangular systems
are efficiently solved using this factorization. The preconditioner for the whole system is block diagonal, each block being
a SPAI or MUMPS preconditioner.

\section{Numerical results}
\label{sec:numres}

The purpose of this section is the comparison between the unstable formulation~\eqref{eq:coupledvf}
and the stable formulation~\eqref{eq:weakcouptrans} with the coupling parameter $\eta=1$.

Consider an ellipsoid with major axis directed along the $z$-axis. This object is included inside a larger ball.
The external border of the ball after discretization is the surface $\Gamma_\infty$. A potential flow is computed around the ellipsoid and inside the ball,
such that the flow is uniform outside the ball, of Mach number $0.3$ and directed along the $z$-axis.
An acoustic monopole source lies upstream of the ball, on the $z$-axis as well.
Four different meshes are considered, see Table~\ref{tabmesh}.
For accuracy reasons, a rule of thumb in boundary
  element methods for the classical Helmholtz equation is to set the
  mean edge length to a value eight to ten times smaller than the wavelength of the source.
In our simulations, we first generate the mesh and then apply the
Prandtl--Glauert transformation. With the present choice for the Mach number,
the mesh is at most extended by a factor $\gamma_\infty \approx 1.048$. Moreover, the integral operators are computed at
the transformed wavenumber $\hat{k}_\infty=\gamma_\infty k_\infty
\approx 0.21$~m.
We then verify that, for Mesh~1, the mean length of the edges of the transformed mesh
is eight times smaller than the wavelength. The three coarser meshes do
not satisfy the rule of thumb and are used as comparison supports and in
numerical experiments requiring a large number of resolutions.

From Table~\ref{tabmesh}, for fine meshes, the number of basis functions used to discretize the unknown $p$ in the variational
formulation~\eqref{eq:weakcouptrans} takes a smaller
part in the total number of basis functions than for coarse meshes (from 20\% down to 12\%).
Therefore, the relative complexity added to
\eqref{eq:coupledvf} by the third equation of~\eqref{eq:weakcouptrans} decreases with the total number of unknowns, which
is an interesting property when it comes to industrial test cases.
Figure~\ref{mesh1} displays Mesh 1 and the rescaled velocity $\vec{M}_0$ of the potential flow.

In what follows, a frequency $\mathtt{f}$ is called resonant if $-\hat{k}_\infty^2=-\frac{4\pi^2\mathtt{f}^2}{\gamma_\infty^2}\in\Lambda$,
where $\Lambda$ is the set of Dirichlet eigenvalues for the Laplacian on $\mathbb{R}^3\backslash\overline{\Omega^+}$.
The set $\Lambda$ depends on the shape of the coupling surface $\Gamma_\infty$, which slightly changes after each
discretization.

\renewcommand{\arraystretch}{1.5}
\begin{table}
\begin{center}
\begin{tabular}{|c |c |c |c |c |}
 \cline{2-5}
\multicolumn{1}{c|}{} & Mesh 1 & Mesh 2 & Mesh 3 & Mesh 4\\
 \hline
number of volumic dofs $\Phi$ & $1796$ & $687$ & $194$ & $79$ \\
 \hline
number of surfacic $\mathbb{P}_0$ dofs $\lambda$ & $808$ & $510$ & $270$ & $148$\\
 \hline
number of surfacic $\mathbb{P}_1$ dofs $p$ & $406$ & $257$ & $137$ & $76$\\
 \hline
proportion of dofs $p$ in the total number of dofs & $11.9\%$ & $15.0\%$ & $18.6\%$ & $20.0\%$\\
 \hline
smallest edge (mm) & $7.09$ & $8.78$ & $15.71$ & $19.18$\\
 \hline
mean edge (mm) & $22.64$ & $32.20$ & $49.78$ & $66.46$\\
 \hline
largest edge (mm) & $56.87$ & $70.62$ & $103.59$ & $112.71$\\
 \hline
\end{tabular}
\end{center}
\caption{\label{tabmesh} Characteristics of the four considered meshes.}
\end{table}

\begin{figure}[h!]
 \centering
\includegraphics[width=0.47\textwidth]{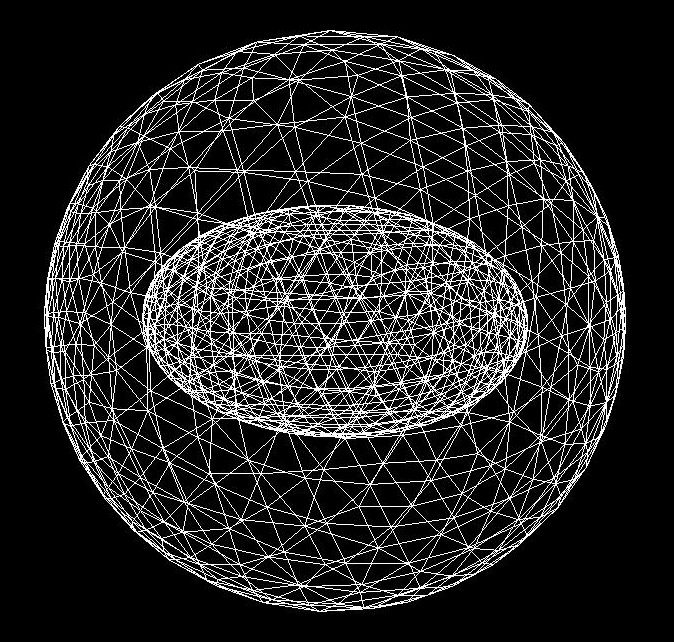}
\includegraphics[width=0.5\textwidth]{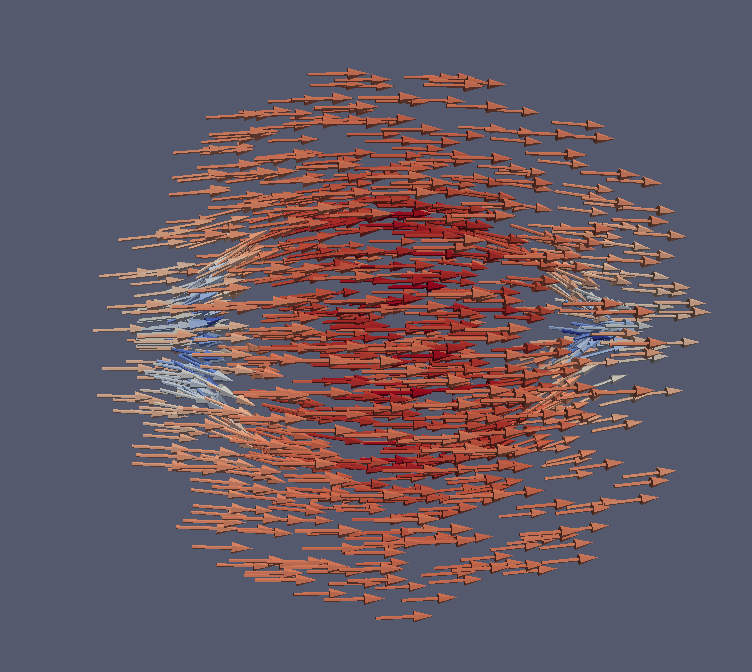}
 \caption{Left: representation of Mesh 1, Right: potential flow around the ellipsoid.}
\label{mesh1}
\end{figure}


\subsection{Comparison of pressure fields}

As seen in Theorem~\ref{theounstab}, the unstable formulation~\eqref{eq:coupledvf} is not well-posed at resonant frequencies.
First, a prospective study to identify a resonant frequency for each of the four meshes is carried out by monitoring
the condition number of the corresponding matrices.
A resonant frequency for Mesh 1, Mesh 2, Mesh 3, and Mesh 4 is identified around $1509.849$~Hz, $1513.431$~Hz, $1521.015$~Hz,
and $1535.704$~Hz, respectively.

The convergence of the iterative solver is monitored by requiring that the Euclidian norm of the relative residual
is smaller than $10^{-6}$.
Additional tests indicate that the discretized solution to the stable formulation does not change much below this value of the
relative residual.
For Mesh 1, away from a resonance, say at $1500$~Hz, the scattered pressure fields computed with the unstable and 
stable formulations are very similar. This holds as well for the total pressure fields, see Figure
\ref{pres1}. At the resonant frequency $1509.849$~Hz, the unstable formulation~\eqref{eq:coupledvf} yields pressure maps quite different
from the ones at $1500$~Hz, whereas the stable formulation~\eqref{eq:weakcouptrans} yields pressure maps very similar to the ones at $1500$~Hz,
see Figure~\ref{pres2}.
The distortion of the scattered field with the unstable
formulation~\eqref{eq:coupledvf} is the result of the significant magnification of numerical errors by the ill-conditioning
of the linear system approximating~\eqref{eq:coupledvf}.


\begin{figure}[h!]
 \centering
\includegraphics[width=0.45\textwidth]{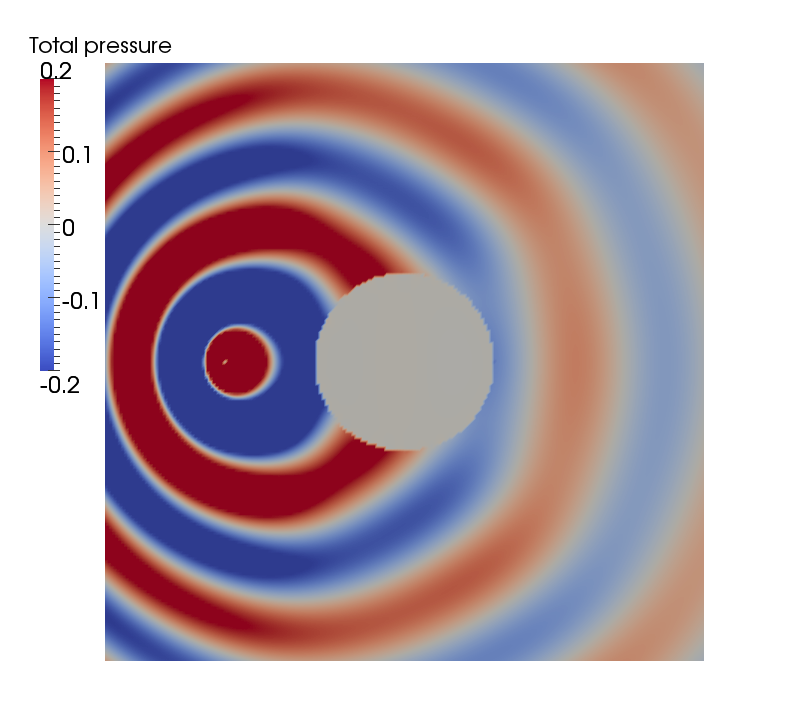}
\includegraphics[width=0.45\textwidth]{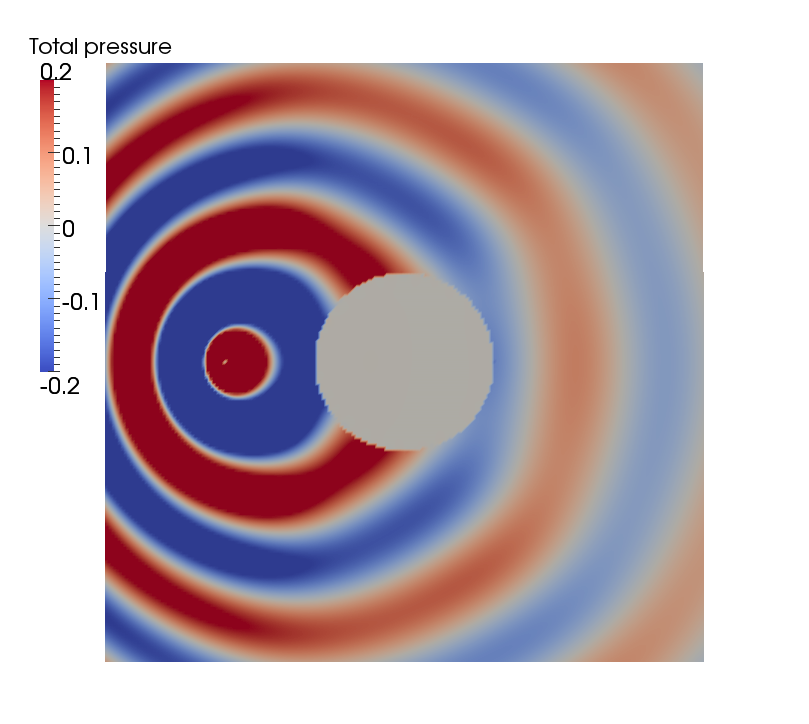}\\
\includegraphics[width=0.45\textwidth]{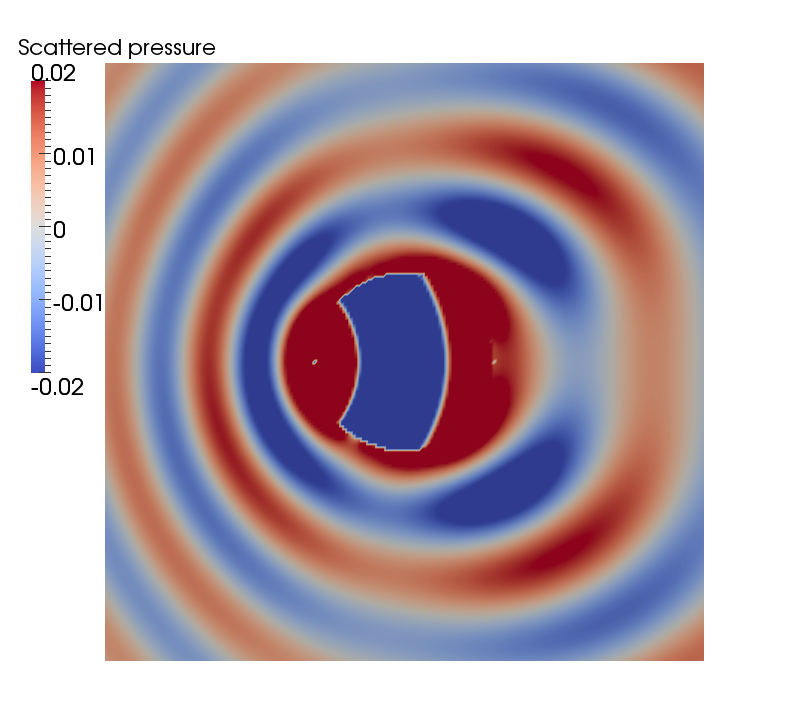}
\includegraphics[width=0.45\textwidth]{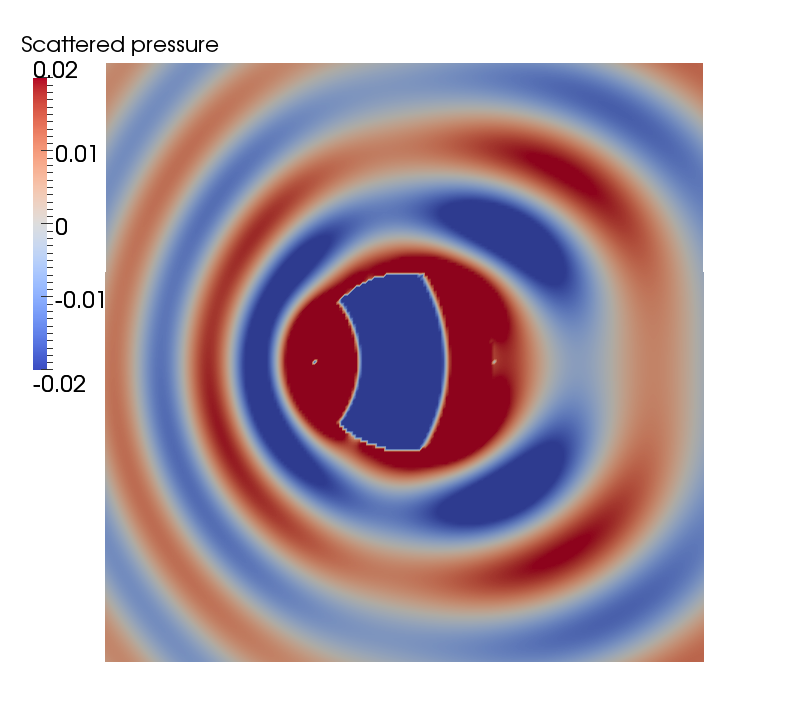}
 \caption{Mesh 1, $1500$~Hz. Top : real part of the total pressure; left: unstable formulation~\eqref{eq:coupledvf},
right: stable formulation~\eqref{eq:weakcouptrans}.
Bottom : real part of the scattered pressure; left: unstable formulation~\eqref{eq:coupledvf},
right: stable formulation~\eqref{eq:weakcouptrans}.
At this non-resonant frequency, both formulations yield similar results.}
\label{pres1}
\end{figure}


\begin{figure}[h!]
 \centering
\includegraphics[width=0.45\textwidth]{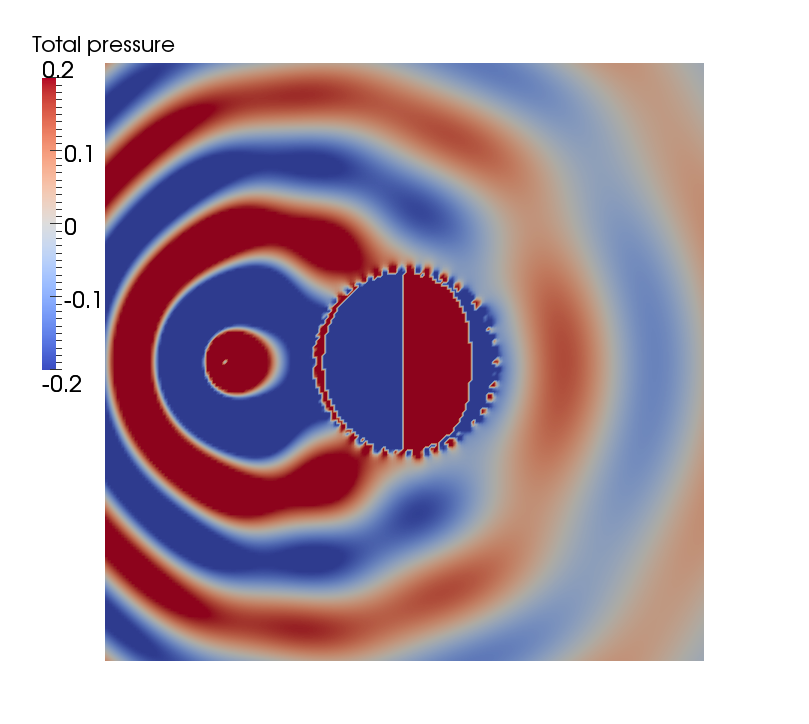}
\includegraphics[width=0.45\textwidth]{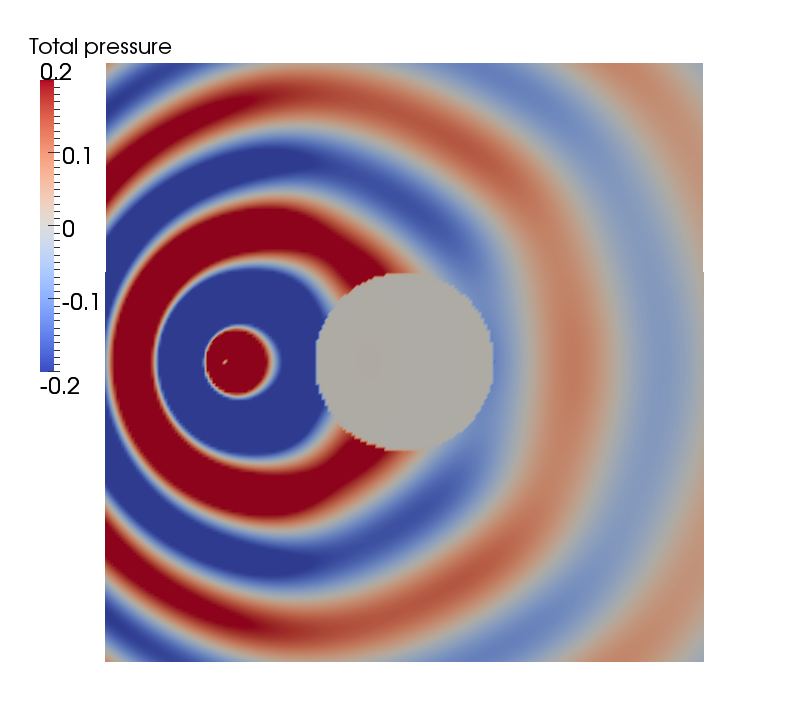}\\
\includegraphics[width=0.45\textwidth]{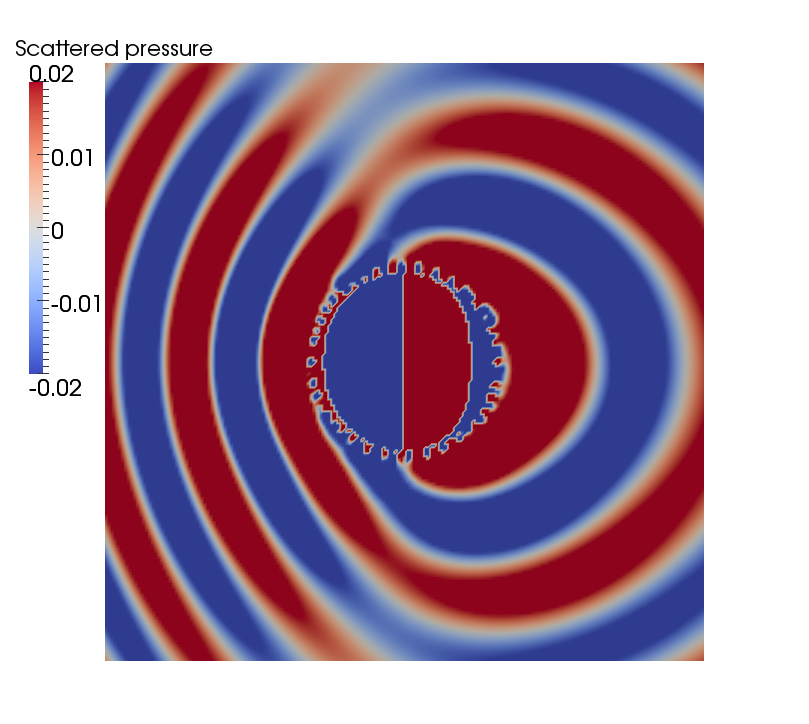}
\includegraphics[width=0.45\textwidth]{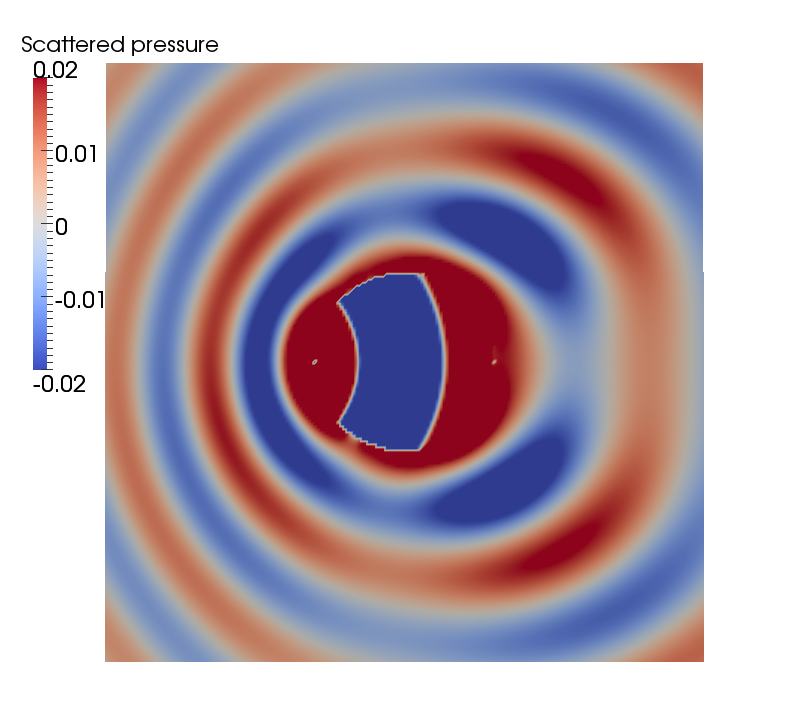}
 \caption{Mesh 1, $1509.849$~Hz. Top : real part of the total pressure; left: unstable formulation~\eqref{eq:coupledvf},
right: stable formulation~\eqref{eq:weakcouptrans}.
Bottom : real part of the scattered pressure; left: unstable formulation~\eqref{eq:coupledvf},
right: stable formulation~\eqref{eq:weakcouptrans}.
At this resonant frequency, the two formulations yield different results.}
\label{pres2}
\end{figure}

\subsection{Auxiliary variable $p$}

In Figure~\ref{fig:dummy1}, the left plot indicates that with Mesh 1,
the magnitude of $p$ is around $0.5\%$ of the scattered pressure.
The right plot shows the behaviour of the magnitude of $p$ (measured as $\|p\|_{L^\infty(\Gamma_\infty)}$)
with respect to the stopping criterion of the iterative solver for
the four meshes.
The finer the mesh, the smaller the auxiliary variable $p$, which is consistent with
the fact that the $p$-component of the solution to~\eqref{eq:weakcouptrans} vanishes (see Section~\ref{sec:stabcoupl}).


\setlength\figureheight{0.43\textwidth}
\setlength\figurewidth{0.45\textwidth}
\begin{figure}[h!]
   \begin{minipage}[c]{.48\linewidth}
\includegraphics[width=1.05\textwidth]{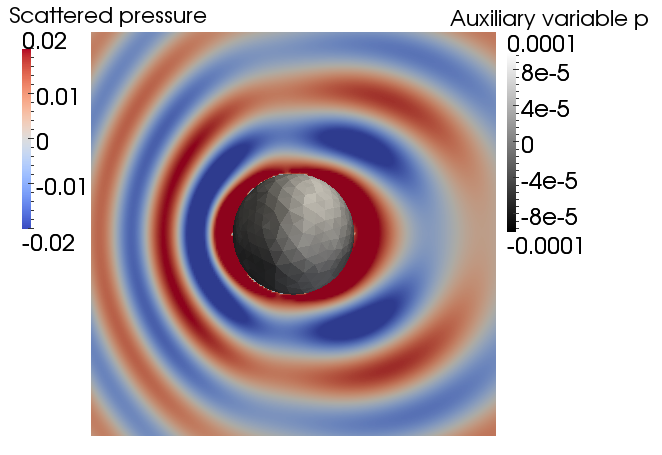}
   \end{minipage} \hfill
   \begin{minipage}[c]{.48\linewidth}
%
%
%
%
\begin{tikzpicture}

\begin{loglogaxis}[
xlabel={stopping criterion of the iterative solver},
ylabel={magnitude of the auxiliary variable $p$},
xmin=1e-06, xmax=0.1,
ymin=0.0001, ymax=0.1,
axis on top,
width=\figurewidth,
height=\figureheight,
legend style={at={(0.03,0.97)}, anchor=north west},
legend entries={Mesh 1,Mesh 2,Mesh 3,Mesh 4}
]
\addplot [green!50.0!black]
coordinates {
(0.1,0.0170174) (0.01,0.001062111) (0.001,0.0001447102) (0.0001,0.0001435521) (1e-05,0.00014751089) (1e-06,0.0001470233) 
};
\addplot [red]
coordinates {
(0.1,0.04252) (0.01,0.001757) (0.001,0.0002871) (0.0001,0.000263) (1e-05,0.0002722) (1e-06,0.0002728) 
};
\addplot [black]
coordinates {
(0.1,0.02401742) (0.01,0.00100374) (0.001,0.000509612) (0.0001,0.000508785) (1e-05,0.000510179) (1e-06,0.000510036) 
};
\addplot [blue]
coordinates {
(0.1,0.0170964) (0.01,0.001480005) (0.001,0.001150991) (0.0001,0.001097344) (1e-05,0.001101953) (1e-06,0.001101317) 
};
\path [draw=black, fill opacity=0] (axis cs:13,1)--(axis cs:13,1);

\path [draw=black, fill opacity=0] (axis cs:1,13)--(axis cs:1,13);

\path [draw=black, fill opacity=0] (axis cs:13,0)--(axis cs:13,0);

\path [draw=black, fill opacity=0] (axis cs:0,13)--(axis cs:0,13);

\end{loglogaxis}

\end{tikzpicture}
   \end{minipage}
 \caption{Stable formulation~\eqref{eq:weakcouptrans} at $1500$~Hz. Left : real part of the scattered pressure and auxiliary variable $p$
with Mesh 1. Right : Magnitude of the auxiliary variable $p$ as a function of the stopping criterion of the iterative linear solver
with all meshes.}
\label{fig:dummy1}
\end{figure}

\subsection{Comparison of condition numbers}

Figure~\ref{fig:narrow} presents the condition numbers of 
the matrices resulting from the formulations~\eqref{eq:coupledvf} and~\eqref{eq:weakcouptrans} as a function of the frequency.
In the left plot, the curves are centered at the resonant frequencies.
The finer the mesh, the higher the condition number explodes. The width
of the peak at the resonance does not appear to depend on the mesh. In the right plot, a larger frequency bandwidth is considered
with Mesh~$2$.
Owing to the frequency sampling (every $5$~Hz), some resonances may be missed, and the local maxima may not be accurately reached
(in particular, from the left plot, the local maximum of $7.2$ for log(cond($M$)) at $1513.431$~Hz is very underestimated).
The stable formulation~\eqref{eq:weakcouptrans} produces somewhat larger condition numbers for the large majority of the
frequencies, but, unlike the unstable formulation~\eqref{eq:coupledvf}, it presents no resonance.
Moreover, from the Weyl formula,
the number of resonant frequencies smaller than $\mathtt{f}$ increases as $\mathtt{f}^{\frac{3}{2}}$, making
the need for a stable formulation even more important for simulations at higher frequencies.


\begin{figure}[h!]
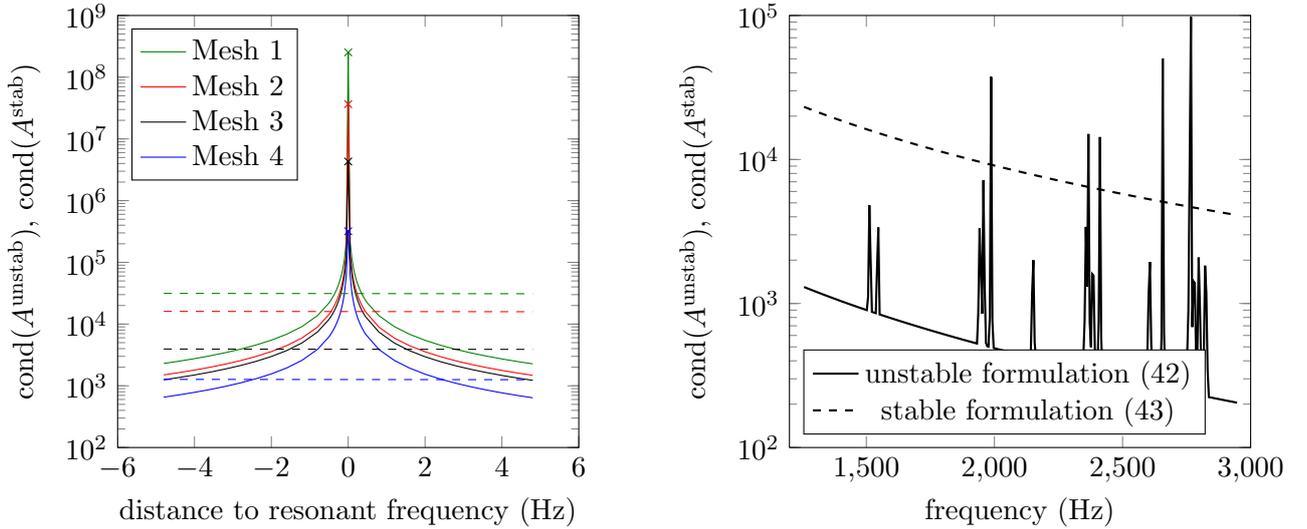

   \begin{minipage}[c]{.48\linewidth}
\include{cond.tikz}
   \end{minipage} \hfill
   \begin{minipage}[c]{.48\linewidth}
\include{condwide.tikz}
   \end{minipage}
 \caption{Condition number of the matrix for the unstable formulation~\eqref{eq:coupledvf} (solid) and the stable formulation
\eqref{eq:weakcouptrans} (dashed). Left: centered representation around a resonant frequency for the four meshes.
Right: larger bandwidth with Mesh 2.}
\label{fig:narrow}
\end{figure}

\subsection{Convergence}

To further study the impact of the ill-conditioning of the unstable formulation~\eqref{eq:coupledvf} on the computed solution,
the preconditioning is not used in what follows. First, the value of the acoustic pressure on a network of $10000$
points located further than $0.5$ m from the center of the sphere (therefore in $\Omega^+$) is computed
using the stable formulation~\eqref{eq:weakcouptrans} with Mesh 1 at the resonant frequency $1509.849$~Hz.
This computed acoustic pressure is called the accurate pressure.
Next, the acoustic pressure on the same network of points is computed for different values of the number of iterations of the
solver, using the unstable formulation~\eqref{eq:coupledvf} and the stable formulation~\eqref{eq:weakcouptrans} with Mesh 1 at the same frequency.
The relative difference between the computed pressure and the accurate pressure in Euclidian norm is called the relative error.
Figure~\ref{fig:conv} presents the relative residual and the relative error with respect to the number of iterations.
With the unstable formulation~\eqref{eq:coupledvf}, the relative residual decreases irregularly. In particular,
it stays constant during around $200$ iterations. 
The relative error decreases, stays constant, rises after $400$ iterations,
and finally stabilizes at a large value, whereas the relative residual keeps converging to zero.
As for every ill-conditioned problem, the relative residual cannot be used
to ascertain convergence towards the correct solution.
In particular, after $600$ iterations, the relative residual is extremely small, while the error is of order one.
With the stable formulation~\eqref{eq:weakcouptrans}, the relative residual and the relative error decrease
regularly, and in the same fashion.


\setlength\figureheight{0.43\textwidth}
\setlength\figurewidth{0.45\textwidth}
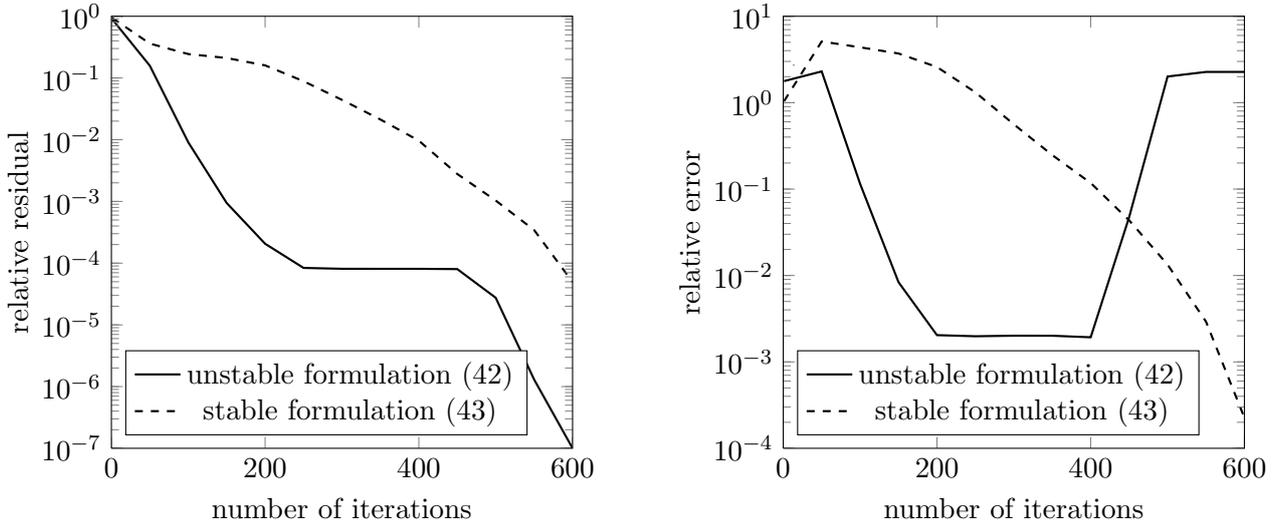
\begin{figure}[h!]
   \begin{minipage}[c]{.48\linewidth}
%
%
%
%
\begin{tikzpicture}

\begin{semilogyaxis}[
xlabel={number of iterations},
ylabel={relative residual},
xmin=0, xmax=600,
ymin=1e-07, ymax=1,
axis on top,
width=\figurewidth,
height=\figureheight,
legend style={at={(0.03,0.03)}, anchor=south west},
legend entries={unstable formulation  \eqref{eq:coupledvf},stable formulation \eqref{eq:weakcouptrans}}
]
\addplot [thick, black]
coordinates {
(1,0.8984425) (50,0.1579681) (100,0.0090716) (150,0.000942) (200,0.0002063) (250,8.37e-05) (300,8.1e-05) (350,8.09e-05) (400,8.09e-05) (450,8.01e-05) (500,2.75e-05) (550,1.3e-06) (600,1e-07) 
};
\addplot [thick, black, dashed]
coordinates {
(1,0.9403811) (50,0.361097) (100,0.2442476) (150,0.2107814) (200,0.1600748) (250,0.088739) (300,0.0442015) (350,0.0212418) (400,0.0096456) (450,0.002784) (500,0.0010342) (550,0.0003435) (600,4.93e-05) 
};
\path [draw=black, fill opacity=0] (axis cs:13,1)--(axis cs:13,1);

\path [draw=black, fill opacity=0] (axis cs:1,13)--(axis cs:1,13);

\path [draw=black, fill opacity=0] (axis cs:13,0)--(axis cs:13,0);

\path [draw=black, fill opacity=0] (axis cs:0,13)--(axis cs:0,13);

\end{semilogyaxis}

\end{tikzpicture}
   \end{minipage} \hfill
   \begin{minipage}[c]{.48\linewidth}
%
%
%
%
\begin{tikzpicture}

\begin{semilogyaxis}[
xlabel={number of iterations},
ylabel={relative error},
xmin=0, xmax=600,
ymin=0.0001, ymax=10,
axis on top,
width=\figurewidth,
height=\figureheight,
legend style={at={(0.03,0.03)}, anchor=south west},
legend entries={unstable formulation  \eqref{eq:coupledvf},stable formulation \eqref{eq:weakcouptrans}}
]
\addplot [thick, black]
coordinates {
(1,1.77556933957) (50,2.29957265782) (100,0.115137303055) (150,0.00837003341556) (200,0.00204025999332) (250,0.00198022326249) (300,0.00200981344827) (350,0.00200795531505) (400,0.00192632906879) (450,0.0472752701018) (500,2.00707584587) (550,2.26908752538) (600,2.26964708358) 
};
\addplot [thick, black, dashed]
coordinates {
(1,1.03838850164) (50,5.10713694066) (100,4.37373932915) (150,3.7082299273) (200,2.58574691448) (250,1.31723549998) (300,0.561265150296) (350,0.247534122656) (400,0.117674867498) (450,0.0435285780663) (500,0.0133601763604) (550,0.0029229648068) (600,0.000225731205522) 
};
\path [draw=black, fill opacity=0] (axis cs:13,1)--(axis cs:13,1);

\path [draw=black, fill opacity=0] (axis cs:1,13)--(axis cs:1,13);

\path [draw=black, fill opacity=0] (axis cs:13,0)--(axis cs:13,0);

\path [draw=black, fill opacity=0] (axis cs:0,13)--(axis cs:0,13);

\end{semilogyaxis}

\end{tikzpicture}
   \end{minipage}
 \caption{Mesh 1 at resonance $1509.849$~Hz; relative residual (left) and relative error (right) with respect to the number of iterations.}
\label{fig:conv}
\end{figure}

\subsection{Choice of the coupling parameter $\eta$}

In the stable formulation~\eqref{eq:weakcouptrans}, the choice of the coupling parameter $\eta$ is expected to have a direct
effect of the condition number of the matrix $A^{\rm stab}$. In Figure~\ref{fig:couplingparam}, this condition number
is plotted for Mesh 4 and for various values of $\eta$. For $\eta=0$,
equations~\eqref{eq:weakcouptrans1}-\eqref{eq:weakcouptrans2} are decoupled from~\eqref{eq:weakcouptrans3}, and
\eqref{eq:weakcouptrans1}-\eqref{eq:weakcouptrans2} become~\eqref{eq:coupledvf}, so that the curve for $\eta=0.001$
is similar to the curve of the unstable formulation for Mesh 4 in Figure~\ref{fig:narrow}. The condition number appears to be
the smallest for $\eta$ in the range 1 to 10, and worsens for lower and higher values of $\eta$. This motivates the choice
$\eta=1$ made in the above simulations.

\setlength\figureheight{0.5\textwidth}
\setlength\figurewidth{0.5\textwidth}
\begin{figure}[h!]
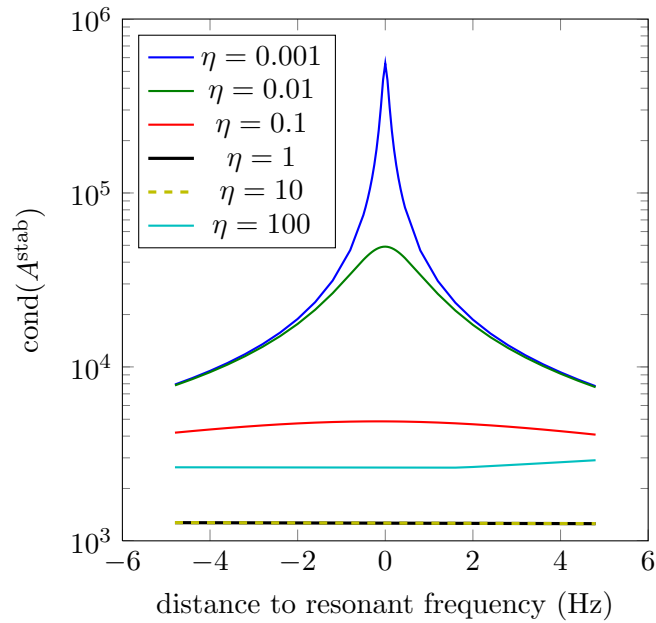

\center
\include{condM.tikz}
 \caption{Condition number of the matrix for the stable formulation~\eqref{eq:weakcouptrans} centered around the resonant
frequency at $1535.704$~Hz for Mesh 4. In this case, the chosen value $\eta=1$ leads to the minimal condition numbers.}
\label{fig:couplingparam}
\end{figure}

\section{Conclusion}
\label{sec:conclusion}

In this work, we derived two coupled formulations for the convected Helmholtz equation with
non-uniform flow in a bounded domain. The first formulation involves one surface unknown and is well-posed except at some resonant frequencies
of the source, while the second formulation is unconditionally well-posed and involves two surface unknowns.
Our numerical results show
that at resonant frequencies, the discretization of the first formulation is ill-conditioned so that the
pressure field is plagued by spurious oscillations. Moreover, the second formulation remains tractable
within large industrial problems since the relative complexity added by
the second surface unknown decreases with the size of the mesh.
The interest in the second formulation is also enhanced by the fact that, at higher frequencies, the density of resonant frequencies is more important.

As long as the uniform flow assumption in the exterior domain is reasonable, more
complex flows in the interior domain can be considered, as well as more complex boundary conditions at the surface of the
scattering object. These extensions only require to modify the finite element part of the present methodology.

Another interesting extension of this work is the resolution of parametrized aeroacoustic problems, with the frequency of the source
as a parameter, using reduced-order models, for instance by means of Proper Generalized Decomposition or Reduced Basis methods.
Using the first formulation may involve ill-conditioned numerical resolutions if the frequency range of interest
contains resonant frequencies, whereas the second formulation guarantees well-posedness of the procedure. Moreover,
the complexity of the online stage of the reduced-order model is not
increased when using the second formulation.

\section*{Acknowledgement}
This work was partially supported by EADS Innovation Works. The authors thank Toufic Abboud (IMACS), Nolwenn Balin
(EADS Innovation Works), Fran\c{c}ois Dubois (CNAM), Patrick Joly (INRIA), and Tony Leli\`evre (CERMICS) for fruitful
discussions.



\bibliographystyle{plain}


\begin{thebibliography}{10}

\bibitem{MUMPS}
P.R. Amestoy, I.S. Duff, and J.-Y. L'Excellent.
\newblock Multifrontal parallel distributed symmetric and unsymmetric solvers.
\newblock {\em Comput. Methods Appl. Mech. Engrg.}, 184(2--4):501--520, 2000.

\bibitem{Amiet}
R.~Amiet and W.~R. Sears.
\newblock The aerodynamic noise of small-perturbation subsonic flows.
\newblock {\em J. Fluid Mech.}, 44:227--235, 1928.

\bibitem{jcp-biblio3}
E.~B\'{e}cache, A.~S. Bonnet-Ben~Dhia, and G.~Legendre.
\newblock Perfectly matched layers for the convected {H}elmholtz equation.
\newblock {\em SIAM J. Numer. Anal.}, 42(1):409--433, 2004.

\bibitem{jcp-biblio6}
M.~Beldi and A.~Maghrebi.
\newblock Some new results for the study of acoustic radiation within a uniform
  subsonic flow using boundary integral method.
\newblock {\em Advanced Materials Research}, 488--489:383--395, 2012.

\bibitem{bettess}
P.~Bettess.
\newblock {\em Infinite Elements}.
\newblock Penshaw Press: Cleadon, Sunderland, U.K., 1992.

\bibitem{BW}
H.~Brakhage and P.~Werner.
\newblock \"{U}ber das {D}irichletsche {A}u{\ss}enraum {P}roblem f\"{u}r die
  {H}elmholtzsche {S}chwingungsgleichung.
\newblock {\em Arch. der Math.}, 16:325--329, 1965.

\bibitem{Brenner}
S.C. Brenner and L.R. Scott.
\newblock {\em The Mathematical Theory of Finite Element Methods}, volume~15 of
  {\em Texts in Applied Mathematics}.
\newblock Springer, 2008.

\bibitem{buffa2}
A.~Buffa and R.~Hiptmair.
\newblock Regularized combined field integral equations.
\newblock {\em Numer. Math.}, 100(1):1--19, 2005.

\bibitem{Carpentieri}
B.~Carpentieri.
\newblock {\em {S}parse preconditioners for dense linear systems from
  electromagnetic applications}.
\newblock PhD thesis, CERFACS, 2002.

\bibitem{precomp}
B.~Carpentieri, I.~Duff, L.~Giraud, and G.~Sylvand.
\newblock Combining fast multipole techniques and an approximate inverse
  preconditioner for large electromagnetism calculations.
\newblock {\em SIAM J. Sci. Comput.}, 27(3):774--792, 2005.

\bibitem{bemfemelasticity}
C.~Carstensen, S.A. Funken, and E.P. Stephan.
\newblock On the adaptive coupling of {FEM} and {BEM} in 2-{D}-elasticity.
\newblock {\em Numer. Math.}, 77(2):187--221, 1997.

\bibitem{Costabel}
M.~Costabel.
\newblock {\em Symmetric methods for the coupling of finite elements and
  boundary elements}, volume~1 of {\em Boundary Elements IX}.
\newblock Springer-Verlag, Berlin, 1987.

\bibitem{bemfemcoupled2}
C.~Dom\'{i}nguez, E.P. Stephan, and M.~Maischak.
\newblock {FE}/{BE} coupling for an acoustic fluid-structure interaction
  problem. {R}esidual a posteriori error estimates.
\newblock {\em Internat. J. Numer. Methods Engrg.}, 89(3):299--322, 2012.

\bibitem{DDMT}
F.~Dubois, E.~Duceau, F.~Mar\'echal, and I.~Terrasse.
\newblock Lorentz transform and staggered finite differences for advective
  acoustics.
\newblock Technical report, EADS, 2002.

\bibitem{Ern}
A.~Ern and J.L. Guermond.
\newblock {\em Theory and Practice of Finite Elements}, volume 159 of {\em
  Applied Mathematical Sciences}.
\newblock Springer, 2004.

\bibitem{fairweather}
G.~Fairweather, A.~Karageorghis, and P.A. Martin.
\newblock The method of fundamental solutions for scattering and radiation
  problems.
\newblock {\em Eng. Anal. Bound. Elem.}, 27(7):759--769, 2003.

\bibitem{Garofalo}
N.~Garofalo and F.-H. Lin.
\newblock Unique continuation for elliptic operators: A geometric-variational
  approach.
\newblock {\em Comm. Pure Appl. Math.}, 40(3):347--366, 1987.

\bibitem{Glauert}
H.~Glauert.
\newblock The effect of compressibility on the lift of an aerofoil.
\newblock {\em Proc. R. Soc. Lond. Ser. A Math. Phys. Eng. Sci.},
  118(779):113--119, 1928.

\bibitem{Goldstein}
M.E. Goldstein.
\newblock {\em Aeroacoustics}.
\newblock McGraw-Hill International Book Company, 1976.

\bibitem{nonlinear}
M.F. Hamilton and D.T. Blackstock.
\newblock {\em Nonlinear Acoustics: Theory and Applications}.
\newblock Elsevier Science \& Tech, 1998.

\bibitem{rhipt1}
R.~Hiptmair.
\newblock Coupling of finite elements and boundary elements in electromagnetic
  scattering.
\newblock {\em SIAM J. Numer. Anal.}, 41(3):919--944, 2003.

\bibitem{hiptmair}
R.~Hiptmair and P.~Meury.
\newblock Stabilized {FEM}-{BEM} coupling for {H}elmholtz transmission
  problems.
\newblock {\em SIAM J. Numer. Anal.}, 44(5):2107--2130, 2006.

\bibitem{introacous}
A.~Hirschberg and S.~W. Rienstra.
\newblock {\em An Introduction to Acoustics}.
\newblock Eindhoven University of Technology, 2004.

\bibitem{Wendland}
G.~C. Hsiao and W.~L. Wendland.
\newblock {\em Boundary Element Methods: Foundation and Error Analysis}.
\newblock John Wiley \& Sons, Ltd, 2004.

\bibitem{jcp-biblio1}
J.M. Jin and V.V. Liepa.
\newblock A note on hybrid finite element method for solving scattering
  problems.
\newblock {\em IEEE Trans. Ant. Prop.}, 36(10):1486--1490, 1988.

\bibitem{Nedelec}
C.~Johnson and J.~C. N\'ed\'elec.
\newblock On the coupling of boundary integral and finite element methods.
\newblock {\em Math. Comp.}, 35(152):1063--1079, 1980.

\bibitem{Langou}
J.~Langou.
\newblock {\em {S}olving large linear systems with multiple right-hand sides}.
\newblock PhD thesis, INSA, 2003.

\bibitem{Leis}
R.~Leis.
\newblock {Z}ur {D}irichletschen {R}andwertaufgabe des {A}u{\ss}enraumes der
  {S}chwingungsgleichung.
\newblock {\em Math. Z.}, 90:205--211, 1965.

\bibitem{levillain}
V.~Levillain.
\newblock {\em {Couplage \'el\'ements finis-\'equations int\'egrales pour la
  r\'esolution des \'equations de Maxwell en milieu h\'et\`erog\`ene}}.
\newblock PhD thesis, \'Ecole Polytechnique, 1991.

\bibitem{bemfemcoupled3}
F.~Leydecker, M.~Maischak, E.P. Stephan, and M.~Teltscher.
\newblock Adaptive {FE}-{BE} coupling for an electromagnetic problem in {$\Bbb
  R\sp 3$}---a residual error estimator.
\newblock {\em Math. Methods Appl. Sci.}, 33(18):2162--2186, 2010.

\bibitem{lighthill1}
M.~J. Lighthill.
\newblock On sound generated aerodynamically. {I}. {G}eneral theory.
\newblock {\em Proc. R. Soc. Lond. Ser. A Math. Phys. Eng. Sci.},
  211(1107):564--587, 1952.

\bibitem{lighthill2}
M.~J. Lighthill.
\newblock On sound generated aerodynamically. {II}. {T}urbulence as a source of
  sound.
\newblock {\em Proc. R. Soc. Lond. Ser. A Math. Phys. Eng. Sci.},
  222(1148):1--32, 1954.

\bibitem{nonlinear1}
M.~Maischak and E.P. Stephan.
\newblock A {FEM}-{BEM} coupling method for a nonlinear transmission problem
  modelling coulomb friction contact.
\newblock {\em Comput. Methods Appl. Mech. Engrg.}, 194(2-5):453--466, 2005.

\bibitem{jcp-biblio2}
B.~McDonald and A.~Wexler.
\newblock Finite-element solution of unbounded field problems.
\newblock {\em IEEE Trans. Microwave Theory Tech.}, 20(12):841--847, 1972.

\bibitem{mclean}
W.~McLean.
\newblock {\em Strongly Elliptic Systems and Boundary Integral Equations}.
\newblock Cambridge University Press, 2000.

\bibitem{jcp-biblio5}
D.~Mitsoudis, C.~Makridakis, and M.~Plexousakis.
\newblock Helmholtz equation with artificial boundary conditions in a
  two-dimensional waveguide.
\newblock {\em SIAM J. Math. Anal.}, 44(6):4320--4344, 2012.

\bibitem{blockgcr}
D.P. O'Leary.
\newblock The block conjugate gradient algorithm and related methods.
\newblock {\em Linear Algebra Appl.}, 29:293--322, 1980.

\bibitem{Panich}
O.~Panich.
\newblock On the question of solvability of the exterior boundary value
  problems for the wave equation and {M}axwell's equations.
\newblock {\em Russian Math. Surv.}, 20:221--226, 1965.

\bibitem{powles}
C.~J. Powles and B.~J. Tester.
\newblock Asymptotic and numerical solutions for shielding of noise sources by
  parallel coaxial jet flows.
\newblock In {\em 14th AIAA/CEAS Aeroacoustics Conference}, 2008.

\bibitem{GMRES}
Y.~Saad and M.~Schultz.
\newblock {GMRES}: A generalized minimal residual algorithm for solving
  nonsymmetric linear systems.
\newblock {\em SIAM J. Sci. and Stat. Comput.}, 7(3):856--869, 1986.

\bibitem{sauter}
S.A. Sauter and C.~Schwab.
\newblock {\em Boundary Element Methods}, volume~39 of {\em Springer Series in
  Computational Mathematics}.
\newblock Springer, 2010.

\bibitem{Tartar}
L.~Tartar.
\newblock {\em An Introduction to Sobolev Spaces and Interpolation Spaces},
  volume~3 of {\em Lecture Notes of the Unione Matematica Italiana}.
\newblock Springer, 2007.

\bibitem{utzmann}
J.~Utzmann, C.-D. Munz, M.~Dumbser, E.~Sonnendr{\"u}cker, S.~Salmon, S.~Jund,
  and E.~Fr{\'e}nod.
\newblock {\em Numerical Simulation of Turbulent Flows and Noise Generation}.
\newblock Springer, 2009.

\bibitem{nonlinear2}
O.~von Estorff and M.~Firuziaan.
\newblock Coupled {BEM}/{FEM} approach for nonlinear soil/structure
  interaction.
\newblock {\em Eng. Anal. Bound. Elem.}, 24(10):715--725, 2000.

\bibitem{zienkiewicz1}
O.~C. Zienkiewicz, D.~W. Kelly, and P.~Bettess.
\newblock The coupling of the finite element method and boundary solution
  procedures.
\newblock {\em Int. J. Numer. Meth. Engng.}, 11:355--375, 1977.

\bibitem{zienkiewicz}
O.C. Zienkiewicz and P.~Bettess.
\newblock Fluid-structure dynamic interaction and wave forces. {A}n
  introduction to numerical treatment.
\newblock {\em Int. J. Numer. Meth. Engng.}, 13(1):1--16, 1978.

\end{thebibliography}

\end{document}